\documentclass[11pt]{article}
\usepackage{amssymb}
\usepackage{tikz}
\usepackage[latin1]{inputenc}
\usepackage{graphicx,color}
\def\nd{\noindent}

\usepackage{color}
\newtheorem{theorem}{Theorem}[section]
\newtheorem{definition}{Definition}[section]
\newtheorem{lemma}{Lemma}[section]
\newtheorem{proposition}{Proposition}[section]
\newtheorem{remark}{Remark}[section]

\newcommand{\fim}{\hfill\rule{2mm}{2mm}}

\date{}
\begin{document}
\title{
\vspace{0.5in} {\bf How to break the uniqueness of $W^{1,p}_{loc}(\Omega)$-solutions for very singular elliptic problems by non-local terms }}

\author{
{\bf\large Carlos Alberto Santos}\footnote{Carlos Alberto Santos acknowledges
the support of CAPES/Brazil Proc.  $N^o$ $2788/2015-02$,}\,\, ~~~~~~ {\bf\large Lais Moreira Santos}
\hspace{2mm}\\
{\it\small Universidade de Bras\'ilia, Departamento de Matem\'atica}\\
{\it\small   70910-900, Bras\'ilia - DF - Brazil}\\
{\it\small e-mails: csantos@unb.br,
lais.santos@ufob.edu.br }\vspace{1mm}\\
}

\date{}
\maketitle \vspace{-0.2cm}

\begin{abstract}
In this paper, we are going to show existence of branches of  bifurcation for positive $W^{1,p}_{loc}(\Omega)$-solutions for the very  singular non-local $\lambda$-problem 
$$
   -{\Big(\int_\Omega g(x,u)dx\Big)^r}\Delta_pu={\lambda } \Big(a(x)u^{-\delta} + b(x)u^{\beta}\Big)   \ \ \mbox{in} \ \ \Omega, \ \ \ \ u > 0 \ \ \ \mbox{in} \ \Omega \ \ \ \mbox{and} \ \ u=0 \ \ \mbox{on} \ \partial \Omega, $$
where $\Omega \subset \mathbb{R}^N $  is a smooth bounded domain, $\delta >0$, $0 <  \beta < p-1$, $a $ and $b$  are non-negative measurable  functions and $g$ is a positive continuous function. 

Our approach is based on sub-supersolutions techniques, fixed point theory, in the study of $ W^{1,p}_{loc}(\Omega)$-topology of a solution application and  a new comparison principle  for sub-supersolutions in  $W^{1,p}_{loc}(\Omega)$ to a problem with $p$-Laplacian operator perturbed by a very singular term at zero and sublinear at infinity. 
\end{abstract}

\nd {\it \footnotesize 2010 Mathematics Subject Classifications: 35J25,  35J62, 35J75, 35J92} {\scriptsize  }\\
\nd {\it \footnotesize Key words}: {\scriptsize Very singular term, Uniqueness, Non-local, Comparison Principle}

\section{Introduction}
\def\theequation{1.\arabic{equation}}\makeatother
\setcounter{equation}{0}

In this paper, we will deal with the existence, non-existence and multiplicity of positive $W^{1,p}_{loc}(\Omega)$-solutions for the singular non-local quasilinear $\lambda$-problem
$$
{(P_\lambda)}~~~~~~\left\{
\begin{array}{l}
-{\Big(\displaystyle\int_\Omega g(x, u)dx\Big)^r\Delta_pu = {\lambda } }\Big(a(x)u^{-\delta} + b(x)u^{\beta}\Big)  ~ \mbox{in } \Omega,\\
    u>0 ~ \mbox{in }\Omega,~~
    u=0  ~ \mbox{on }\partial\Omega, 
\end{array}
\right.
$$
where $ \Omega \subset \mathbb{R}^N ( N \geq 2)$ is a smooth bounded domain, $ \ -\Delta_pu = -\mbox{div}(|\nabla u|^{p-2}\nabla u)$ is the $p$-Laplacian operator, $1<p<N$, $\delta > 0$, $0 < \beta < p-1$,   $\lambda > 0$ is a real parameter and $a, b,g\geq 0$  are appropriate functions.

An overview about $(P_\lambda)$. This problem  is non-local  due to the presence of the 
term $(\int_\Omega g(x, u))^r$, which implies that equation in $(P_{\lambda})$ is no longer pointwise equality. In general, the presence of such terms give rise  some  additional difficulties in approaching this kind of problems by classical arguments. For example, many non-local problems are non-variational, in the sense that techniques of variational methods can not be applied in a direct way. The non-local problems have been extensively studied in recent years and their applications arise in various contexts, for example, in the studies of systems of particles in thermodynamical equilibrium
via gravitational potential (\cite{ALY},  \cite{KR2}), 2-D fully turbulent behavior of real flow \cite{CAG}, thermal runaway in Ohmic heating (\cite{LAC}, \cite{CAR}), physics of plasmas, population behavior \cite{COR}, thermo-electric flow in a conductor \cite{LAC3}, gravitational equilibrium of polytropic stars  \cite{KR}, modeling of cell aggregation through interaction with a chemical  \cite{WO} and others.

Many authors have studied non-local problems,
   but up to this date there are no results in the literature in the direction of the $p$-Laplacian operator with $p \neq 2$ in the context of $W^{1,p}_{loc}(\Omega)$-solutions to singular ones. About related problems with weak singularities ($0<\delta <1$) for Laplacian operator, we quote the works  \cite{CLA1, MA, REN, WANG} that showed existence of positive solutions. We note that just in \cite{CLA1} was considered weak solutions still in $H_0^1{(\Omega)}$, while the others ones treated the problem in the context of classical solutions.
   
Recently, Souza at all \cite{souza} considered 
$$
\left\{
\begin{array}{l}
-{\displaystyle g\Big(x, \int_{\Omega} u^p\Big)\Delta u = {\lambda } }u^{\beta}  ~ \mbox{in } \Omega,\\
    u>0 ~ \mbox{in }\Omega,~~
    u=0  ~ \mbox{on }\partial\Omega, 
\end{array}
\right.
$$
where $0<\beta \leq 1$ and $g$ satisfies suitable assumptions. They showed how the structure of the branches of bifurcation of the problem  is affected by the non-local term both with $g$ depending on  $x\in \Omega$ and in the autonomous case as well.

Despite  Garc\'ia-Meli\'an and Lis \cite{MEL} have not  studied neither a singular problem nor a Dirichlet boundary condition problem, we are going to highlight their techniques. They showed existence of solution to the blow-up problem
   \begin{eqnarray}\label{arcoya}
   {\Big(1 + \frac{1}{|\Omega|}\int_{\Omega} g(u)dx\Big) }\Delta u= \lambda {f(u)}  \ \ \mbox{in}  \   \mbox{in} \  \Omega,  \ u > 0 \ \ \mbox{in} \  \Omega, \  \  u= \infty\ \ \mbox{on}  \ \ \partial\Omega,
   \end{eqnarray} 
where $f:[0,\infty) \to (0,\infty)$ is an appropriate continuous function, by  decoupling (\ref{arcoya}) in the  system 
\begin{equation}\label{arcoya2}
\left\{\begin{array}{l}
\Delta u=  \alpha f(u) \ \ \mbox{in}  \ \ \Omega, \  \  u= \infty \ \ \mbox{on}  \ \ \partial\Omega \\
 \alpha = \lambda \Big(1 + \frac{1}{|\Omega|}\int_{\Omega} g(u) dx\Big)^{-1}
\end{array}\right.
\end{equation}
and studying the behavior of the pair $(\alpha,u)$ solution of (\ref{arcoya2}).

To obtain  branches of bifurcation in $(0, \infty)\times \Vert\cdot \Vert_{\infty}$, we have inspired on the Garc\'ia-Meli\'an and Lis' strategy by exploring the $(0,\infty)\times W^{1,p}_{loc}(\Omega)$-topology of the pair $(\alpha, u)$ by using a new Comparison Principle for $W_{loc}^{1,p}(\Omega)$-sub and supersolutions that we proved as well. Taking advantage of this approach, we present a complete picture of the  bifurcation diagram of Problem $(P_\lambda)$. In particular,  we show how the presence of the non-local term changes the structure of the bifurcation of the local problem that emanates from $(0,0)$ and bifurcates from  infinity at infinity.
 
 Before introducing the main results of this work, we need to clarify what we mean by Dirichlet boundary condition and solution to  $(P_{\lambda})$.
 
After the  remarkable paper of Mckenna \cite{LAZ}, in 1991,  we know that a solution of the problem $(P_\lambda)$ with $a=1$,  $b=0$ and $p=2$ still lies in $H_0^1(\Omega)$ if, and only if, $0<\delta< 3$. So, for stronger  singularities, we need of a more general concept of zero-boundary conditions. 

\begin{definition}\label{fronteira}
We say that $u\leq 0$ on $\partial \Omega$ if $ (u - \epsilon)^+  \in W_0^{1,p}(\Omega)$ for every  $ \epsilon > 0 $ given. Furthermore, $u\geq 0$ if $-u\leq 0$ and $u = 0$ on $\partial \Omega$ if $u$ is non-negative and non-positive  on $\partial \Omega$. 
\end{definition}

About solutions.
\begin{definition}\label{sf}
We say that $u $ is a  $ W^{1,p}_{loc}(\Omega)$-solution for $({P_{\lambda}})$ if $u>0$ in $\Omega$ $($for each $K \subset \subset \Omega$ given there exists a positive constant $c_{K}$ such that $u \geq c_{K} > 0$ in $K$) and
\begin{eqnarray}\label{sol}  
\Big(\displaystyle\int_{\Omega} g(x, u)dx\Big)^{r}\displaystyle\int_{\Omega} |\nabla  u|^{p-2}\nabla u\nabla \varphi dx = {\lambda} \displaystyle\int_{\Omega} \Big(a(x)u^{-\delta} + b(x)u^{\beta}\Big)\varphi dx~\mbox{for all }   \varphi  \in  C_{c}^{\infty}(\Omega). 
\end{eqnarray}
\end{definition}

Our approach is based on issues about existence, uniqueness and $(\alpha, u_\alpha)$-behavior in the $(0,\infty)\times W^{1,p}_{loc}(\Omega)$-topology  for the local problem
$$ (L_{\alpha})~~~~
\left\{
\begin{array}{l}
  -\Delta_p u=\alpha \Big( a(x) u^{-\delta} + b(x)u^{\beta} \Big)   ~\mbox{in } \Omega,\\
    u>0    ~\mbox{in } \partial\Omega,~~
    u>0    ~\mbox{on }  \Omega,
 \end{array}
\right. 
$$
where $(L_{\alpha})$ is the problem $(P_\lambda)$ with $\lambda=\alpha$ and $r=0$.

In this sense, we refine the proofs of existence of $W_{loc}^{1,p}(\Omega)$-solutions found \cite{ORSINA}, \cite{CAN} and \cite{ELVES} to include both more general potentials  $a$ and $b$ and a bigger variation of $p$.  The more delicate issue is the uniqueness of solutions in $W_{loc}^{1,p}(\Omega)$ for the problem $(L_{\alpha})$.  The main results in \cite{CAN1} and \cite{CAN} treated about this issue. In \cite{CAN1}, by exploring the linearity of the Laplacian operator, they showed uniqueness of solutions  to  $(P_\lambda)$ with $p=2$, $b=0$ and $a\in L^1(\Omega)$, while in \cite{CAN} the problem $(P_\lambda)$ with $b=0$ was treated with some restrictions either on the potential $a$ or on the geometry of the domain.

Despite the next result being so classical, it is new even for Laplacian operator both by generality of the potentials $a$ and $b$ and principally by the uniqueness of solution in the context of $W_{loc}^{1,p}(\Omega)$ for very singular nonlinearities perturbed by $(p-1)$-sublinear ones. 

\begin{theorem}\label{T6}
Assume  $b \in L^{(\frac{p^*}{\beta+1})'}(\Omega)$. If one of the below assumptions holds 
\begin{itemize}
\item [$(h)_1$:] $0 < \delta < 1 $ and $a \in L^{(\frac{p^*}{1-\delta})'}(\Omega)$;
\item [$(h)_2$:]$\delta \geq 1$ and $a \in L^1(\Omega),$  
\end{itemize}
then there exists a $u=u_{\alpha} \in W_{loc}^{1,p}(\Omega)$
solution to the problem $(L_{\alpha})$ for each $\alpha>0 $ given. Moreover, if $\delta \leq 1$, then $u \in W_0^{1,p}(\Omega)$. Besides this, the solution is {\bf only} if $ a + b > 0$.
\end{theorem}

As said above, by using  fine properties of the solution $u_\alpha \in  W_{loc}^{1,p}(\Omega)$ together with a Loc-Schmitt's Theorem \cite{LOC}, we able to prove the next result.

Before stating it, let us consider
\begin{itemize}
\item[($\tilde{h}_0$)] $a, b  \in  L^m(\Omega) $ for some  $ m > N/p$,
\item[($\tilde{h}_1$)]  $a, b  \in  L^m(\Omega) $ for some  $ m > N$
\end{itemize}
and denote by
$$\Sigma = \{(\lambda, u) \in (0,\infty)\times C(\overline{\Omega})~/~u \in W^{1,p}_{loc}(\Omega) \cap C(\overline{\Omega})~\mbox{is a solution of}~ (P_{\lambda})\}.$$

So, we have.
\begin{theorem} 
Assume  $\delta > 0$ and $0<\beta < p-1$ hold. If:
\begin{enumerate} 
\item[$1)$] $g \in C(\overline{\Omega}\times  [0, \infty), (0, \infty))$, for some $f_1 \in C(\overline{\Omega})$
\begin{eqnarray}\label{4}
\displaystyle\lim_{t\rightarrow \infty}g(x,t)t^{\theta_1} = f_1(x) > 0  ~\mbox{uniformly in } ~\overline{\Omega}, 
\end{eqnarray} 
and in addition 
\begin{itemize}
\item[$a)$]  $(\tilde{h}_0)$  and \ $\theta_1r < p - 1 - \beta $ hold, then $(P_\lambda)$ admits at least one solution in $\Sigma$ for each $\lambda > 0$ given. Besides this, we can assume $f_1 \equiv\infty$ if $r \geq 0$ and $ f_1 \equiv 0$ if $r<0$,
\item[$b)$]  $(\tilde{h}_1)$, \ $\theta_1r  > p - 1 - \beta $ and $\theta_1 < 1$ hold, then there exists   $ 0<\lambda^*< \infty$ such that  $(P_{\lambda})$  admits at least two $W_{loc}^{1,p}(\Omega)\cap C(\overline{\Omega})$-solutions for each $ \lambda\in (0,\lambda^*)$ given,  at least one solution for $\lambda = \lambda^*$ and no solution for $\lambda > \lambda^*$. Moreover, we can admit $f_1\equiv 0$ if $r \geq 0$ and $ f_1 \equiv\infty$ if $r<0$,
\end{itemize}
\item[$2)$] $g \in C((0, \infty), (0, \infty))$,  $(\tilde{h}_1)$ is satisfied, for some $f_1$ and $f_2 \in C(\overline{\Omega}) $
\begin{eqnarray}\label{5}
\displaystyle\lim_{t\rightarrow \infty}g(x, t)t^{\theta_1} = f_1(x) > 0 ~ \mbox{and} ~\displaystyle\lim_{t\rightarrow 0^+}g(x, t)t^{\theta_2} = f_2(x) > 0 ~\mbox{uniformly in} ~\overline{\Omega}
\end{eqnarray} 
and additionally 
\begin{itemize}
\item[a)]   $ \theta_1r < p - 1 - \beta$,  \ $\theta_2r > p -1 + \delta $ and $  \theta_2 < 1 $ hold, then
there exists a  $0<\lambda^*<\infty$  such that  \  $(P_\lambda)$ admits at least two $W_{loc}^{1,p}(\Omega)\cap C(\overline{\Omega})$-solutions for $\lambda>\lambda^*$,  at least one for $\lambda=\lambda^*$ and no solutions for  $0<\lambda<\lambda^*$. Besides this, we can have $f_1 \equiv f_2 \equiv \infty$ if $r > 0$ and $f_1 \equiv f_2 \equiv 0$ if $r < 0$,
\item[b)]  $ \theta_1r > p - 1 - \beta$,  \ $\theta_2r > p -1 + \delta $ and $ \theta_1,  \theta_2 < 1 $ hold,  then $(P_{\lambda})$ admits at least one $W_{loc}^{1,p}(\Omega)\cap C(\overline{\Omega})$-solution for each $\lambda > 0$ given. In this case, we can have $f_1 \equiv 0$ and $f_2 \equiv \infty$ if $r > 0$ and $f_1\equiv \infty$ and $f_2 \equiv 0$ if $r < 0$. 
\end{itemize}
\end{enumerate}
Besides these, in all cases $\Sigma$ is the {\it continuum} of solutions given by a curve that:
\begin{enumerate}
\item[$(i)$] emanates from $0$ at $\lambda=0$ and bifurcates from infinity  at $\lambda=\infty$ in the case $1-a)$,
\item[$(ii)$]  emanates from $0$ at $\lambda=0$ and bifurcates from infinity  at $\lambda=0$ in the case $1-b)$,
\item[$(iii)$]  emanates from $0$ at $\lambda=\infty$ and bifurcates from infinity  at $\lambda=\infty$ in the case $2-a)$,
\item[$(iv)$]  emanates from $0$ at $\lambda=\infty$ and bifurcates from infinity  at $\lambda=0$ in the case $2-b)$,
\end{enumerate}
\end{theorem}

Below, we draw the $(0,\infty)\times \Vert \cdot \Vert_{\infty}$-diagram of $W_{loc}^{1,p}(\Omega)$-solutions given by the above Theorem.
    \vspace{0.5cm}
     \newpage
     \begin{figure}
     \centering
	\begin{tikzpicture}[scale=.55]
	\draw[thick, ->] (-1, 0) -- (8, 0);
	\draw[thick, ->] (0, -1) -- (0, 7);
    \draw[thick] (0, 0) .. controls (1, 3) and (4, 0) ..(7,6.5);
	\draw (8,0) node[below]{$\lambda$};
	\draw (0, 0) node[below left]{$0$};
	\draw (0,7) node[left]{$\|u\|_\infty$};
   \draw (3,-1) node[below]{\textrm{Fig. 1 Theorem 1.2 item 1-$a)$}};
    \end{tikzpicture}
    \hspace{1.0cm}
  	\begin{tikzpicture}[scale=0.55]
   \draw[thick, ->] (-1, 0) -- (8, 0);
	\draw[thick, ->] (0, -1) -- (0, 7);
	\draw[thick] (0.0 ,0.0) .. controls (7.0,3.0) and (0.2,3) .. (0.15, 6.8);
	\draw[thick, dotted] (3.17,0) -- (3.17, 7.0);
	\draw (8,0) node[below]{$\lambda$};
	\draw (3.17,0) node[below]{$\lambda^*$};
	\draw (0, 0) node[below left]{$0$};
	\draw (0,7) node[left]{$\|u\|_\infty$};
   \draw (3,-1) node[below]{\textrm{Fig. 2 Theorem 1.2 item 1-$b)$}};
    \end{tikzpicture}
       \hspace{1.0cm}
       \begin{tikzpicture}[scale=0.55]
	\draw[thick, ->] (-1, 0) -- (8, 0);
	\draw[thick, ->] (0, -1) -- (0, 7);
    \draw[thick] (7, 7) .. controls (0.1, 5) and (2, 0.2) ..(7,0.1);
	\draw[thick, dotted] (2.5,0) -- (2.5,7);
	\draw (8,0) node[below]{$\lambda$};
	\draw (2.5,0) node[below]{$\lambda^*$};
	\draw (0, 0) node[below left]{$0$};
	\draw (0,7) node[left]{$\|u\|_\infty$};
\draw (3,-1) node[below]{\textrm{Fig. 3 Theorem 1.2 item 2-$a)$}};
	\end{tikzpicture}
	 \hspace{1.0cm}
			\begin{tikzpicture}[scale=0.55]
\draw[thick, ->] (-1, 0) -- (8, 0);
	\draw[thick, ->] (0, -1) -- (0, 7);
	\draw[thick] (0.2,7) .. controls (1,1) and (5, 0.2
	) .. (7,0.2) ;
		\draw (8,0) node[below]{$\lambda$};	
	\draw (0, 0) node[below left]{$0$};
	\draw (0,7) node[left]{$\|u\|_\infty$};
\draw (3,-1) node[below]{\textrm{Fig. 3 Theorem 1.2 item 2-$b)$}};
	\end{tikzpicture}	
\end{figure}	
          

Below, we list some of the main contributions of this work to literature:
\begin{enumerate}
\item our result of uniqueness for the local problem $(L_{\alpha})$ improves the main theorems in \cite{CAN1} and \cite{CAN} by:
\begin{enumerate}
\item[$(i)$] removing any requirement about the geometry of the domain,
\item[$(ii)$] permitting a perturbation of the very singular term by a sublinear one,
\item[$(ii)$] including more general potentials $a$ and $b$,
\end{enumerate}
\item singular problems of the type $(P_{\lambda})$ involving the $p$-Laplacian operator with $  \delta $  assuming any positive value and weights $ a $ and $ b $ being unbounded have not yet been considered in the literature up to now,
\item Theorem 1.2 complements the principal results in \cite{souza} by consider a perturbation of their nonlinearity by a strong singular one,
\item the problem $(P_\lambda)$ with the non-local term $a(t) = t^r $, $t>0$ with $r\in \mathbb{R}$ for a singular pertubation by a $(p-1)$-sublinear ones has not yet been considered in literature so far. Our non-local term is not requireded being either bounded from below by positive constant or from above, and in fact it may be singular at $t=0$. See for instance \cite{MA}, \cite{REN}, \cite{souza} and references therein.
\end{enumerate}

This work has the following structure. In Section 2, we prove the uniqueness of $W_{loc}^{1,p}(\Omega)$-solutions to the problem $(L_{\alpha})$ inspired on ideas of \cite{Saa} and \cite{CAN}. To do this a  comparison principle for $W_{loc}^{1,p}(\Omega)$-sub and supersolutions is established. As the proof of existence of solution of  Theorem \ref{T6} follows by a refinement of   well-know arguments, we will  just sketch it in the Appendix. In Section 3, by exploring  the uniqueness of  $W_{loc}^{1,p}(\Omega)$-solutions to Problem $(L_{\alpha})$,   appropriate test functions together with a result of Boccardo and Murat \cite{MUR}, we are able to prove that the operator $ T:(0,\infty) \to W_{loc}^{1,p}(\Omega)$ (see (\ref{1141}) below) is well-defined and continuous. In Section 4, we conclude the proof of Theorem 1.2. 

\vspace{0.3cm} 
Throughout this paper, we make use of the following notations: 
\begin{itemize}
\item The norm in $L^p(\Omega)$ is denoted by $\|.\|_p$. 
\item The space $W_0^{1,p}(\Omega)$ endowed with the norm $\|\nabla u\|_p^p =  \displaystyle\int_{\Omega} |\nabla u|^pdx. $
\item $\vert U \vert$ stands for the Lebesgue measure of mensurable set $U \subset \mathbb{R}^N$. 
\item $C_c^{\infty}(\Omega) = \{ u : \Omega \rightarrow \mathbb{R}\} | u \in C^{\infty}(\Omega) \ \mbox{e} \ supp ~ u \subset \subset \Omega \}$.
\item $L_c^{\infty}(\Omega) = \{ u: \Omega \rightarrow \mathbb{R}\} | u \in L^{\infty}(\Omega) \ \mbox{e} \ supp ~ u \subset \subset \Omega \}.$
\item $C, C_1, C_2, \cdots $  denote positive constants. 

\end{itemize}

\section{Comparison Principle for Sub and Supersolutions in $W_{loc}^{1,p}(\Omega)$}

Below, let us define  subsolution and  supersolution to the problem $(L_1)$, that is, to the problem 
\begin{equation}\label{lsp} \left\{
\begin{array}{l}
  -\Delta_p u=  a(x) u^{-\delta} + b(x)u^{\beta}  ~\mbox{in } \Omega,\\
    u>0    ~\mbox{in } \partial\Omega,~~
    u>0    ~\mbox{on }  \Omega.
 \end{array}
\right.
\end{equation}

\begin{definition}\label{D3} A function $\underline{v} \in W_{loc}^{1,p}(\Omega) $ is a subsolution of $(\ref{lsp})$ if:
\begin{itemize}
\item[$i)$]  there is a positive constant $c_K$ such that $\underline{v} \geq c_K$ in $K$ for each $K \subset \subset \Omega$ given; 
\item[$ii)$] the inequality
\begin{eqnarray}\label{3.9}
 \displaystyle\int_{\Omega} |\nabla \underline{v}|^{p-2}\nabla \underline{v} \nabla \varphi dx \leq   \displaystyle\int_{\Omega} \Big(\frac{a(x) }{\underline{v}^{\delta}} + b(x)\underline{v}^{\beta}\Big) \varphi dx
 \end{eqnarray}
 holds for all $ 0 \leq \varphi \in C_{c}^{\infty}(\Omega)$. 
 When $\overline{v}\in W_{loc}^{1,p}(\Omega) $ satisfies  the reversed inequality in $(\ref{3.9})$, it is called a supersolution to problem  $(\ref{lsp})$. 
\end{itemize}
\end{definition}

\begin{theorem}[$W_{loc}^{1,p}(\Omega)$-Comparison Principle]\label{unicidade1} Suppose that  $b \in L^{(\frac{p^*}{\beta +1})'}(\Omega)$ and $a+b>0$ in $\Omega$. Assume that one of the below assumptions
\begin{itemize}
\item [$(h)_1$:] $0 < \delta < 1 $ and $a \in L^{(\frac{p^*}{1-\delta})'}(\Omega)$;
\item [$(h)'_2$:]$\delta > 1$ and $a \in L^1(\Omega),$ 
\item [$(h)_3$:]$\delta = 1$ and $a \in L^s(\Omega)$ for some $s>1$ 
\end{itemize}
holds. If $\underline{v}, \overline{v}  \in W_{{loc}}^{1,p}(\Omega)$ are subsolution and supersolution of $(\ref{lsp})$, respectively, with  $\underline{v} \leq 0$ in $\partial \Omega$, then $\underline{v} \leq \overline{v}$ a.e. in $\Omega$. Besides this, if  in addition $\underline{v} $,  $\overline{v} \in W_0^{1,p}(\Omega)$ and (\ref{3.9}) is satisfied for all $0 \leq \varphi \in W_0^{1,p}(\Omega)$, then the same conclusion hold even for $a \in L^1(\Omega)$ in $(h)_3$.
\end{theorem}

To prove Theorem \ref{unicidade1}, let us  consider the functional $J_{\epsilon} : W_0^{1,p}(\Omega) \rightarrow \mathbb{R}$ defined by $$J_{\epsilon}(\omega) = \frac{1}{p}\displaystyle\int_{\Omega} |\nabla \omega|^{p}dx - \displaystyle\int_{\Omega} F_{\epsilon}(x, \omega)dx, $$
for each $\epsilon > 0$ given,
 and denote by ${\cal{C}}$ the convex and closed set
  $${\cal{C}} = \{ \omega \in W_0^{1,p}(\Omega)~/~ 0 \leq \omega \leq \overline{v}\},$$
 where 
 $$F_{\epsilon}(x, \omega) = \displaystyle\int_0^{\omega}\Big[a(x)(t+\epsilon)^{-\delta} + b(x)(t + \epsilon)^{\beta}\Big]dt$$
 and $\overline{v} \in W_{loc}^{1,p}(\Omega)$ is a supersolution  to the problem (\ref{lsp}).

 \begin{lemma}\label{L2}  If $b \in L^{(\frac{p^*}{\beta +1})'}(\Omega)$ and one of the hypotheses $(h)_1$, $(h)'_2$  or $(h_3)$ holds, then the functional $J_{\epsilon}$ is coercive and  weakly lower semicontinuous  on $\cal{C}$.
\end{lemma}
\begin{proof} Set  $\omega \in \cal{C} $. First, we note  that if  $(h_3)$ holds, then there exists  a $C_{\epsilon} >0$ such that $ln| z + \epsilon| \leq C_{\epsilon} (z + \epsilon)^t$ for all $z \geq 0$ and for  $ t =  \min\{{p^*}/{s'}, p-1\}>0$ fixed. So, by  using either this fact,  $(h)_1$ or $(h)'_2$ and Sobolev embedding, we obtain 
$$ J_{\epsilon}(\omega) \geq \left\{\begin{array}{l} \frac{1}{p}\|\nabla \omega \|_p^p - C\Big[ \|a\|_{(\frac{p^*}{1-\delta})'}\|\omega\|_{p^*}^{1-\delta} + \|b\|_{(\frac{p^*}{1+\beta})'}\|\omega\|_{p^*}^{\beta +1} + 1\Big]  \ \ \mbox{if} \ \ 0 < \delta <  1, \\
\frac{1}{p}\|\nabla \omega\|_p^p - C\Big[\|a\|_{s}\|\omega\|_{p^*}^t + \|b\|_{(\frac{p^*}{1+\beta})'}\|\omega \|_{p^*}^{\beta +1} + 1\Big]  \ \ \mbox{if} \ \ \delta = 1,  \\
\frac{1}{p}\|\nabla \omega\|_p^p - C\Big[\|b\|_{(\frac{p^*}{1+\beta})'}\|\omega \|_{p^*}^{\beta +1} + 1\Big]  \ \ \mbox{if} \ \ \delta > 1  
 \end{array}\right.$$  
 that leads to the coerciveness of  $J_{\epsilon}$ in all the cases.
 
Below, let us show that $J_{\epsilon}$ is weakly lower semicontinuous  on $\cal{C}$. Let $(\omega_n) \subset \cal{C}$ such that $\omega_n \rightharpoonup \omega$ in $W_0^{1,p}(\Omega)$.
Suppose first that $0 < \delta < 1$. We claim  that 
 \begin{eqnarray}\label{vitali} \displaystyle\int_{\Omega} \displaystyle\int_0^{\omega_n} a(x)(s+\epsilon)^{-\delta}dtdx \stackrel{n \rightarrow \infty}{\longrightarrow} \displaystyle\int_{\Omega} \displaystyle\int_0^{\omega} a(x)(s+\epsilon)^{-\delta}dtdx. \end{eqnarray} 
 In fact, since  $a \in L^{\left(\frac{p^*}{1 - \delta}\right)'}(\Omega)$,  it follows from the absolute continuity of the Lebesgue integral that for each $\epsilon' > 0$ there exists $\delta > 0$ such that 
$$ \displaystyle\int_A a(x)^{\frac{p^*}{p^* + \delta -1}}dx \leq \Big(\frac{\epsilon'}{C_1}\Big)^{\frac{p^*}{p^* + \delta -1}}$$
for each measurable subset $A$ of $\Omega$ such that  $|A| < \delta$.
 
So, the boundedness of $(\omega_n)$ in $L^{p^*}(\Omega)$ together with the above information lead us to 
$$ \displaystyle\int_A a(x)(\omega_n + \epsilon)^{1 - \delta}dx \leq \Big(\displaystyle\int_A a(x)^{\frac{p^*}{p^* + \delta -1}}dx\Big)^{\frac{p^* + \delta -1}{p^*}}\Big(\displaystyle\int_{\Omega} (\omega_n + \epsilon)^{p^*}dx \Big)^{\frac{1 - \delta}{p^*}}  \leq \epsilon',  $$
that is, $(\omega_n)$ uniformly integrable over $\Omega$. If $\delta=1$, we can redo the above arguments. So,  in both cases our claim follows by applying Vitali's Convergence Theorem.  In the case  $ \delta >1$, the convergence (\ref{vitali}) follows from the classical Lebesgue's Theorem.

Following close arguments as above, we obtain that 
 $$ \displaystyle\int_{\Omega} \displaystyle\int_0^{\omega_n} b(x)(s+\epsilon)^{\beta}dtdx \stackrel{n \rightarrow \infty}{\longrightarrow} \displaystyle\int_{\Omega} \displaystyle\int_0^{\omega} b(x)(s+\epsilon)^{\beta}dtdx $$
as well. This is enough to finish the proof of the Lemma. 
 \end{proof}
  \fim
 \vspace{0.2cm}
 
 Since $\cal{C}$ is convex and closed in the $W_0^{1,p}(\Omega)$-topology,  it follows from Lemma \ref{L2} that there exists a $\omega_0 \in \cal{C}$ such that
$$J_{\epsilon}(\omega_0) = \displaystyle\inf_{\omega  \in \cal{C}}J_{\epsilon}(\omega). $$

\begin{lemma}\label{lema3.5} For all $\varphi \geq 0$ {in} $ C_c^{\infty}(\Omega),$ we have $$\displaystyle\int_{\Omega} |\nabla \omega_0|^{p-2}\nabla \omega_0 \nabla \varphi dx \geq \displaystyle\int_{\Omega} \Big[ a(\omega_0 + \epsilon)^{-\delta} + b(\omega_0 + \epsilon)^{\beta}\Big] \varphi dx. $$
\end{lemma}

\noindent\begin{proof} First, given a non-negative $\varphi \in  C_c^{\infty}(\Omega), $ let us define $v_t: = \min\{\omega_0 + t\varphi, \overline{v}\}$ and \linebreak $\omega_t := (\omega_0 + t\varphi - \overline{v})^+$ for $t>0$. So,  it follows from $\omega_0 \leq \overline{v}$ that  $v_t = \omega_0$ and $\omega_t = 0 $ in $\Omega \backslash supp~\varphi$. From these, we have $v_t \in \cal{C}$, because  $ \overline{v} \in W^{1,p}(supp ~ \varphi)$ and $0 \leq v_t \leq\overline{v}$.
Besides this, since  $\overline{v}>0$ (see definition \ref{D3}),  we can find a $t > 0$ enough small such that $t\varphi \leq 2\overline{v}  - \omega_0$, that is, $\omega_t \in \cal{C}$ as well.

We define $\sigma : [0, 1] \rightarrow \mathbb{R}$ by $\sigma(s) = J_{\epsilon}\Big(sv_t +(1-s)\omega_0\Big). $ Then
 \begin{eqnarray*} 0 & \leq & \displaystyle\lim_{s \rightarrow 0^+} \frac{\sigma(s) - \sigma(0)}{s} = \displaystyle\lim_{s \rightarrow 0^+} \frac{J_{\epsilon}\Big(sv_t +(1-s)\omega_0\Big) - J_{\epsilon}(0) }{s} \\
 & = & \displaystyle\int_{\Omega}|\nabla \omega_0|^{p-2}\nabla \omega_0 \nabla ( v_t - \omega_0)dx - \displaystyle\int_{\Omega}a(x)(\omega_0 + \epsilon)^{-\delta}(v_t - \omega_0)dx - \displaystyle\int_{\Omega}b(x)( \omega_0 + \epsilon)^{\beta}(v_t - \omega_0)dx. 
 \end{eqnarray*}
Hence, using  that $v_t - \omega_0 = t\varphi -\omega_t, $ by the previous inequality we get
\begin{eqnarray}\label{3.12}
 0 & \leq & t \displaystyle\int_{\Omega} \Big[|\nabla \omega_0|^{p-2}\nabla \omega_0\nabla \varphi - a(x)(\omega_0 + \epsilon)^{-\delta}\varphi - b(x)(\omega_0 + \epsilon)^{\beta}\varphi \Big]dx \nonumber \\
 & & - \displaystyle\int_{\Omega} \Big[|\nabla \omega_0|^{p-2}\nabla \omega_0\nabla \omega_t - a(x)(\omega_0 + \epsilon)^{-\delta}\omega_t - b(x)(\omega_0 + \epsilon)^{\beta}\omega_t \Big]dx.
 \end{eqnarray}
However, since $\overline{v}$  is a supersolution of (\ref{lsp}) and $0\leq \omega_t  \in W_0^{1,p}(\Omega) \cap L^{\infty}_{loc}(\Omega)$ (note that $\omega_t \leq t\varphi$),  by the classical density arguments one obtains
\begin{eqnarray}\label{1.12} \displaystyle\int_{\Omega} |\nabla \overline{v}|^{p-2}\nabla \overline{v} \nabla \omega_tdx \geq \displaystyle\int_{\Omega} \Big(a(x)\overline{v}^{-\delta} + b(x)\overline{v}^{\beta}\Big)\omega_t dx. \end{eqnarray}
Dividing both the sides of (\ref{3.12}) by $t > 0$ and using (\ref{1.12}), one obtains
\begin{eqnarray}\label{112}
0 &\leq & \displaystyle\int_{\Omega} \Big[ |\nabla\omega_0|^{p-2}\nabla \omega_0 \nabla \varphi - a(x)(\omega_0 + \epsilon)^{-\delta}\varphi - b(x)(\omega_0  + \epsilon)^{\beta}\Big]dx \nonumber\\
& & + \frac{1}{t}\displaystyle\int_{\Omega} \Big(|\nabla \overline{v}|^{p-2}\nabla \overline{v} - |\nabla \omega_0|^{p-2}\nabla \omega_0 \Big)\nabla \omega_t dx \\
& & + \frac{1}{t} \displaystyle\int_{\Omega} \Big[a(x)\Big((\omega_0 +\epsilon)^{-\delta} - \overline{v}^{-\delta}\Big) + b(x)\Big((\omega_0 +\epsilon)^{\beta} - \overline{v}^{\beta} \Big)\Big] \omega_t dx.\nonumber
 \end{eqnarray}
Below, let us estimate the two last integral in (\ref{112}). First, by using $\omega_t \rightarrow 0$ when ${t \rightarrow 0^+}$, the limit $ |supp ~\omega_t| \stackrel{t \rightarrow 0^+}{\longrightarrow} 0$ and the  monotonicity of the $p$-Laplacian operator, we obtain
\begin{eqnarray*} & &\frac{1}{t}\displaystyle\int_{\Omega} \Big( |\nabla \omega_0|^{p-2}\nabla \omega_0 - |\nabla \overline{v}|^{p-2}\nabla \overline{v}\Big)\nabla \omega_t dx \geq \displaystyle\int_{supp \ \omega_t}  \Big( |\nabla \omega_0|^{p-2}\nabla \omega_0 - |\nabla \overline{v}|^{p-2}\nabla \overline{v}\Big)\nabla \varphi dx \stackrel{t \rightarrow 0}{\longrightarrow} 0. 
\end{eqnarray*}

To last integral, noting that $\omega_0 \leq \overline{v}$, we have 
\begin{eqnarray*}
& & \frac{1}{t}\displaystyle\int_{supp \ \omega_t} \Big[a(x)\Big(\overline{v}^{-\delta} - (\omega_0 +\epsilon)^{-\delta} \Big) + b(x)\Big(\overline{v}^{\beta}  - (\omega_0 +\epsilon)^{\beta} \Big)\Big] \omega_t dx \\
& \geq & -\displaystyle\int_{supp \ \omega_t} \Big[ a(x)\Big|\overline{v}^{-\delta} - (\omega_0 +\epsilon)^{-\delta}\Big|  + b(x)\Big|\overline{v}^{\beta}  - (\omega_0 +\epsilon)^{\beta}\Big|\Big] \varphi dx \stackrel{t \rightarrow 0}{\longrightarrow} 0.
\end{eqnarray*}
Hence, by using these information in (\ref{112}), we conclude the proof. 
\fim
  \end{proof}
\smallskip

\noindent\textbf{\textit{Proof of Theorem 2.1-Conclusion:}}
 Let us set 
 $$\Omega_{\epsilon} := \{ x \in \Omega ~/~ \underline{v}(x) > \omega_0(x) + \epsilon\} ~~\mbox{and}~~ \Omega_{\epsilon}^n = \Omega_{\epsilon} \cap \{ x \in \Omega ~/~ \underline{v}(x)  < n\}$$
 for each $\epsilon>0$ given and $n\in \mathbb{N}$. So,  $\Omega_{\epsilon} = \bigcup_{n \in \mathbb{N}} \Omega_{\epsilon}^n. $ 
 
Assume that $|\Omega_{\epsilon}| >0$ for some $\epsilon>0$. Then it is clear that  $|\Omega_{\epsilon}^n| > 0 $ for all $n\geq n_0'$ for some $n_0' \in \mathbb{N}$, because  $ \Omega_{\epsilon}^n  \subset \Omega_{\epsilon}^{n +1}. $ Let us fix one of this $n$. We claim that there exists a ball $B_{R}(x_0) \subset \subset \Omega$ such that $|B_{R}(x_0) \cap \Omega_{\epsilon}^n| > 0. $ Indeed,  from the compactness of $ \overline{\Omega }$, we can find an open set $ B\subset \mathbb{R}^N$ such that  $|B \cap \Omega_{\epsilon}^n| > 0$. Denote this by $|B \cap \Omega_{\epsilon}^n| = 2\delta>0$. If $B \cap \partial \Omega \neq \emptyset, $ set $A_{\epsilon_0} = \{x \in \Omega ~/~ dist(x, \partial \Omega) < \epsilon_0\}$, where $\epsilon_0 > 0$  is taken in such a way that $|B \cap A_{\epsilon_0}| < \delta. $ In this case, $|\overline{B \cap A_{\epsilon_0}^C} \cap \Omega_{\epsilon}^n| > \delta$. So, our claim follows from the fact that $\overline{B \cap A_{\epsilon_0}^C}$ is a compact  set.

Set $\phi \in C_c^{\infty}(\Omega, [0, 1])$ such that  $supp \ \phi \subset B_{R + r}(x_0)$, $\phi = 1$ in $ B_R(x_0)$ and $|\nabla \phi| \leq Cr^{-t}$ in $B_{R + r}(x_0) \backslash B_R(x_0)$  for an appropriate $t>0$, that will be determined later. So, it is a consequence of this construction that 
$0\neq \varphi_1,\varphi_2 \in L^p(\Omega)\cap L_c^{\infty}(\Omega)$, where $v_n := \min\{\underline{v}, n\} $,
\begin{center}{$ \varphi_1 :=  {\phi\Big[v_n^p - (\omega_0 + \epsilon)^p \Big]^+}{v_n^{1-p}}$  \  and \  
$ \varphi_2 :=  {\phi\Big[v_n^p - (\omega_0 + \epsilon)^p \Big]^+}{(\omega_0 + \epsilon)^{1-p}}. $}
\end{center}

Thus,
\begin{eqnarray*}
\nabla \varphi_1 = \Big[ \frac{v_n^p - (\omega_0 + \epsilon)^{p}}{v_n^{p-1}}\Big]^+ \nabla \phi  + \phi \Big[\nabla v_n - p\frac{(\omega_0 + \epsilon)^{p-1}}{v_n^{p-1}} \nabla (\omega_0 + \epsilon) + (p-1)\frac{(\omega_0 + \epsilon)^p}{v_n^p} \nabla v_n \Big] \chi_{[v_n \geq \omega_0 + \epsilon]}
\end{eqnarray*}
and
\begin{eqnarray*}
\nabla \varphi_2 = \Big[\frac{
v_n^p - (\omega_0 + \epsilon)^{p}}{(\omega_0 + \epsilon)^{p-1}}\Big]^+ \nabla \phi + \phi \Big[\frac{pv_n^{p-1}}{(\omega_0 + \epsilon )^{p-1}}\nabla v_n -  \nabla (\omega_0 + \epsilon) - (p-1)\frac{v_n^p}{(\omega_0 + \epsilon)^p} \nabla (\omega_0 + \epsilon) \Big] \chi_{[v_n \geq \omega_0 + \epsilon]},
\end{eqnarray*}
which  lead us to conclude that $|\nabla \varphi_1| ,|\nabla \varphi_2| \in L^p(\Omega) $, because  $0 < c_K \leq v_n  \leq n $ in $K = supp ~\phi$.

Since $\varphi_1 $, $\varphi_2 \geq 0$ and 
$\varphi_1 $, $\varphi_2 \in W_0^{1,p}(\Omega) \cap L_c^{\infty}(\Omega)$, we get by densities arguments that
\begin{eqnarray*}
\displaystyle\int_{\Omega} |\nabla \underline{v}|^{p-2}\nabla\underline{v} \nabla \varphi_1 dx \leq \displaystyle\int_{\Omega}\Big(a(x) \underline{v}^{-\delta} + b(x)\underline{v}^{\beta}\Big)\varphi_1 dx
\end{eqnarray*}
and
\begin{eqnarray*}
\displaystyle\int_{\Omega} |\nabla \omega_0|^{p-2}\nabla \omega_0 \nabla \varphi_2dx \geq \displaystyle\int_{\Omega} \Big[a(x)(\omega_0 + \epsilon)^{-\delta} + b(x)(\omega_0 + \epsilon)^{\beta}\Big]\varphi_2 dx
\end{eqnarray*}
hold, where $\omega_0$ is given in Lemma \ref{lema3.5}.

So, by calculating and using the above inequalities, we obtain
\begin{eqnarray*}
& & \displaystyle\int_{\Omega} |\nabla \underline{v}|^{p-2}\nabla \underline{v} \nabla \phi \Big[ \frac{v_n^p - (\omega_0 + \epsilon)^{p}}{v_n^{p-1}}\Big]^+ dx + \displaystyle\int_{[\underline{v}  \leq n]} |\nabla \underline{v}|^{p-2}\nabla \underline{v} \nabla \Big[\frac{\underline{v}^p - (\omega_0 + \epsilon)^p}{\underline{v}^{p-1}}\Big]^+ \phi dx \\
& &- p\displaystyle\int_{[\underline{v} > n]} |\nabla \underline{v}|^{p-2}\nabla \underline{v}  \Big[\frac{(\omega_0 + \epsilon)^{p-1}\nabla( \omega_0 + \epsilon)}{n^{p-1}}\Big] \chi_{[\omega_0 + \epsilon < n]}\phi dx \leq \displaystyle\int_{\Omega} \Big(a(x)\underline{v}^{-\delta} + b(x)\underline{v}^{\beta}\Big)\varphi_1 dx
\end{eqnarray*}
and 
\begin{eqnarray*}
& & \displaystyle\int_{\Omega} |\nabla \omega_0|^{p-2}\nabla\omega_0 \nabla \phi \Big[ \frac{v_n^p - (\omega_0 + \epsilon)^{p}}{(\omega_0 + \epsilon)^{p-1}}\Big]^+ dx + \displaystyle\int_{[\underline{v} \leq n]} |\nabla \omega_0|^{p-2}\nabla \omega_0 \nabla \Big[\frac{\underline{v}^p - (\omega_0 + \epsilon)^p}{(\omega_0 + \epsilon)^{p-1}}\Big]^+ \phi dx \\
& & + \displaystyle\int_{[\underline{v} > n]} |\nabla \omega_0|^{p-2}\nabla \omega_0 \nabla \Big[\frac{n^p - (\omega_0 + \epsilon)^p}{(\omega_0 + \epsilon)^{p-1}}\Big]^+ \phi dx \geq \displaystyle\int_{\Omega} \Big[a(x)(\omega_0 + \epsilon)^{-\delta} + b(x)(\omega_0 + \epsilon)^{\beta}\Big]\varphi_2 dx.
\end{eqnarray*}
Hence, by combining the previous inequalities we have
\begin{eqnarray*} 
& & \displaystyle\int_{\Omega} |\nabla\underline{v}|^{p-2}\nabla\underline{v} \nabla \phi\Big[ \frac{v_n^p - (\omega_0 + \epsilon)^{p}}{v_n^{p-1}}\Big]^+ dx + \displaystyle\int_{[\underline{v} \leq n]} |\nabla \underline{v}|^{p-2}\nabla \underline{v} \nabla \Big[\frac{\underline{v}^p - (\omega_0 + \epsilon)^p}{\underline{v}^{p-1}}\Big]^+ \phi dx \\
& & -p\displaystyle\int_{[\underline{v} > n]} |\nabla \underline{v}|^{p-2}\nabla \underline{v}  \Big[\frac{(\omega_0 + \epsilon)^{p-1}\nabla( \omega_0 + \epsilon)}{n^{p-1}}\Big] \chi_{[\omega_0 + \epsilon < n]}\phi dx - \displaystyle\int_{\Omega} |\nabla \omega_0|^{p-2}\nabla\omega_0 \nabla \phi \Big[ \frac{v_n^p - (\omega_0 + \epsilon)^{p}}{(\omega_0 + \epsilon)^{p-1}}\Big]^+ dx \\
& & - \displaystyle\int_{[\underline{v}  \leq n]} |\nabla \omega_0|^{p-2}\nabla \omega_0 \nabla \Big[\frac{\underline{v}^p - (\omega_0 + \epsilon)^p}{(\omega_0 + \epsilon)^{p-1}}\Big]^+ \phi dx   
 - \displaystyle\int_{[\underline{v} > n]} |\nabla \omega_0|^{p-2}\nabla \omega_0 \nabla \Big[\frac{n^p - (\omega_0 + \epsilon)^p}{(\omega_0 + \epsilon)^{p-1}}\Big]^+ \phi dx  \\
& = & \displaystyle\int_{\Omega} |\nabla \underline{v}|^{p-2}\nabla \underline{v} \nabla \varphi_1 dx - \displaystyle\int_{\Omega} |\nabla \omega_0|^{p-2}\nabla \omega_0 \nabla \varphi_2 dx  \\
& \leq & \displaystyle\int_{\Omega} a(x) \Big[ \frac{\underline{v}^{-\delta}}{v_n^{p-1}} - \frac{(\omega_0 + \epsilon)^{-\delta}}{(\omega_0 + \epsilon)^{p-1}} \Big][v_n^p - (\omega_0 + \epsilon)^p]^+\phi dx \\
& & + \displaystyle\int_{\Omega}  b(x) \Big[ \frac{\underline{v}^{\beta}}{v_n^{p-1}} - \frac{(\omega_0 + \epsilon)^{\beta}}{(\omega_0 + \epsilon)^{p-1}} \Big][v_n^p - (\omega_0 + \epsilon)^p]^+\phi dx.
\end{eqnarray*}

Now, by the previous inequality, the next one
$$ -\displaystyle\int_{[\underline{v}  > n]} |\nabla \omega_0|^{p-2} \nabla \omega_0 \nabla \Big[\frac{n^p - (\omega_0 + \epsilon)^p}{(\omega_0 + \epsilon)^{p-1}}\Big]\phi dx  =  \displaystyle\int_{[\underline{v}> n]}|\nabla \omega_0|^p\Big[ 1 + \frac{ n^p(p - 1)}{(\omega_0 + \epsilon)^{p}}\Big]dx  \geq 0 
$$ 
and by the classical Picones's inequality, we get
\begin{eqnarray} \label{3.13}
0 & {\leq} & \displaystyle\int_{[\underline{v} \leq n ]} |\nabla \underline{v}|^{p-2}\nabla \underline{v} \nabla \Big[ \frac{\underline{v}^p - (\omega_0 + \epsilon)^{p}}{\underline{v}^{p-1}}\Big]^+\phi dx 
 -   \displaystyle\int_{[\underline{v}  \leq n]} |\nabla \omega_0|^{p-2}\nabla \omega_0 \nabla \Big[\frac{\underline{v}^p - (\omega_0 + \epsilon)^p}{(\omega_0 + \epsilon)^{p-1}}\Big]^+ \phi dx \nonumber\\
& \leq & \frac{p}{n^{p-1}}\displaystyle\int_{[\underline{v}  > n]} |\nabla \underline{v}|^{p-1}|\nabla \omega_0|( \omega_0 + \epsilon)^{p-1} \chi_{[\omega_0 + \epsilon < n]}\phi dx +  \displaystyle\int_{B_{R+r}\backslash B_R} |\nabla \underline{v}|^{p-1}|\nabla \phi|\Big[ \frac{v_n^p - (\omega_0 + \epsilon)^p}{v_n^{p-1}} \Big]^+dx \nonumber\\
& & + \displaystyle\int_{B_{R+r} \backslash B_R} |\nabla \omega_0|^{p-1}|\nabla \phi| \Big[ \frac{v_n^p - (\omega_0 + \epsilon)^p}{(\omega_0 + \epsilon)^{p-1}}\Big]^+ dx  \nonumber\\
 &  & + \displaystyle\int_{\Omega} a(x) \Big[ \frac{\underline{v}^{-\delta}}{v_n^{p-1}} - \frac{(\omega_0 + \epsilon)^{-\delta}}{(\omega_0 + \epsilon)^{p-1}} \Big][v_n^p - (\omega_0 + \epsilon)^p]^+\phi dx \nonumber\\
 &  & + \displaystyle\int_{\Omega} b(x) \Big[ \frac{\underline{v}^{\beta}}{v_n^{p-1}} - \frac{(\omega_0 + \epsilon)^{\beta}}{(\omega_0 + \epsilon)^{p-1}} \Big][v_n^p - (\omega_0 + \epsilon)^p]^+\phi dx.
\end{eqnarray}

Below, let us estimate the integrals in (\ref{3.13}). To the last two integrals, we can deduce by the assumption $a+b>0$, the inequality $\underline{v}^{-\delta} \leq v_n^{-\delta}$ for all $n \in \mathbb{N}$ and Lebesgue's Theorem, that
 $$\displaystyle\int_{\Omega} a(x) \Big[ \frac{\underline{v}^{-\delta}}{v_{n_0}^{p-1}} - \frac{(\omega_0 + \epsilon)^{-\delta}}{(\omega_0 + \epsilon)^{p-1}} \Big][v_{n_0}^p - (\omega_0 + \epsilon)^p]^+\phi dx + \displaystyle\int_{\Omega} b\Big( \frac{\underline{v}^{\beta}}{v_{n_0}^{p-1}} - \frac{(\omega_0 + \epsilon)^{\beta}}{(\omega_0 + \epsilon)^{p-1}}\Big)\Big[v_{n_0}^{p} - (\omega_0 + \epsilon)^p\Big]^+\phi dx < -4\epsilon', $$
 holds for some  $\epsilon' > 0$ and $n_0 >1$ large.

Now, lets consider the first integral in the second line. 
We claim that $|[\underline{v} > n]| \stackrel{n \rightarrow \infty}{\longrightarrow} 0. $  Indeed, otherwise would exist $\delta' > 0$  and a subsequence $\mathbb{N}' \subset \mathbb{N}$ such that $|[(\underline{v} - \epsilon)^+ > n -\epsilon]| = |[\underline{v} > n]| > \delta'$ for all $n \in \mathbb{N}'$. By using Hypothesis $(\underline{v} - \epsilon)^+ \in W_0^1(\Omega)$, we would have
$$(n -\epsilon)\delta' < \displaystyle\int_{[(\underline{v} - \epsilon)^+ > n -\epsilon]}(\underline{v} - \epsilon)^+dx \leq \displaystyle\int_{\Omega}(\underline{v} - \epsilon)^+ dx \leq C \Vert \nabla  (\underline{v} - \epsilon)^+ \Vert_p < \infty, \ \forall n \in \mathbb{N}', $$
which is absurd. Therefore, from $|[\underline{v} > n]| \stackrel{n \rightarrow \infty}{\longrightarrow} 0 $ we are able to choose  an  $n_0\geq n_0'$ sufficiently large, such that
$$ \Big|{ p\displaystyle\int_{[\underline{v} > n_0]} |\nabla \underline{v}|^{p-2}\nabla \underline{v}  \Big[\frac{(\omega_0 + \epsilon)^{p-1}\nabla( \omega_0 + \epsilon)}{{n_0}^{p-1}}\Big] \chi_{[\omega_0 + \epsilon < n_0]}\phi  }dx \Big| \leq \Big(\displaystyle\int_{[\underline{v} > n_0]} |\nabla \underline{v}|^p \phi^p\Big)^{\frac{p-1}{p}}\|\nabla \omega_0\|_p \leq \epsilon'.$$ 

To the first integral on the ring ${B_{R + r} \backslash B_R }$, we note that the choice of $\phi$ lead us to obtain that
\begin{eqnarray*}
 \displaystyle\int_{B_{R + r} \backslash B_R } |\nabla \underline{v}|^{p-1} |\nabla \phi| \Big[\frac{v_{n_0}^p - (\omega_0 + \epsilon)^p}{v_{n_0}^{p-1}}\Big]^+ dx & \leq & \displaystyle\int_{B_{R + r} \backslash B_R } |\nabla \underline{v}|^{p-1} |\nabla \phi|n_0 dx \\
 & \leq & Cn_0 \|\nabla \phi\|_{L^p(B_{R + r} \backslash B_R)} \\
 & \leq & Cn_0 r^{-t} |B_{R + r} \backslash B_R|^{\frac{1}{p}} \leq C_1n_0 r^{-t + \frac{1}{p}}
 \end{eqnarray*}
 holds for any $t>0$. By taking a $ t < {1}/{p}$, we can choose $ r > 0$ sufficiently small such that
 $$ \displaystyle\int_{B_{R + r} \backslash B_R } |\nabla \underline{v}|^{p-1} |\nabla \phi| \Big[\frac{v_{n_0}^p - (\omega_0 + \epsilon)^p}{v_{n_0}n^{p-1}}\Big]^+ dx < \epsilon'. $$
In a similar way, we can infer that  
 $$ \displaystyle\int_{B_{R + \delta} \backslash B_R} |\nabla \omega_0|^{p-1} |\nabla \phi| \Big[ \frac{v_{n_0}^p - (\omega_0 + \epsilon)^p}{(\omega_0 + \epsilon)^{p-1}}\Big]^+ dx < \epsilon'$$ 
as well. Hence, getting back to the inequality (\ref{3.13}) and using the above informations, we get 
 \begin{eqnarray*} 
 0  \leq \displaystyle\int_{[\underline{v} \leq n_0]} |\nabla \underline{v}|^{p-2}\nabla \underline{v} \nabla \Big(\frac{\underline{v}^p - (\omega_0 + \epsilon)^p}{\underline{v}^{p-1}}\Big)\phi dx - \displaystyle\int_{[\underline{v} \leq n_0]} |\nabla \omega_0|^{p-2}\nabla \omega_0 \nabla \Big(\frac{\underline{v}^p - (\omega_0 + \epsilon)^p}{(\omega_0 + \epsilon)^{p-1}}\Big)\phi dx < 0, 
 \end{eqnarray*} 
 which is an absurd. Therefore $|\Omega_{\epsilon}^n| = 0$ for all $n$, which implies $|\Omega_{\epsilon}| =  0$ and so  $ \underline{v} \leq \omega_ 0 + \epsilon \leq \overline{v} + \epsilon$ a.e in $\Omega$ for all $\epsilon > 0$, whence $\underline{v} \leq {\overline{v}}$.
 
To finish the proof, let us assume that $\underline{v} $,  $\overline{v} \in W_0^{1,p}(\Omega)$, (\ref{3.9}) is satisfied for all $0 \leq \varphi \in W_0^{1,p}(\Omega)$ and $(\underline{v} - \overline{v})^+ \neq 0$. By defining $\underline{v}_n^{\epsilon}(x) := \min\{\underline{v}(x) +\epsilon, n\}$,  $\overline{v}_n^{\epsilon}(x) :=\min\{\overline{v}(x)+\epsilon, n\}$ and the test functions
$$  \varphi_1 = \Big[(\underline{v}_n^{\epsilon})^p - (\overline{v}_n^{\epsilon})^p\Big]^+(\underline{v}_n^{\epsilon})^{1-p} ~\mbox{and } ~ \varphi_2 = \Big[(\underline{v}_n^{\epsilon})^p - (\overline{v}_n^{\epsilon})^p\Big]^+(\overline{v}_n^{\epsilon})^{1-p},$$ 
 we obtain
\begin{eqnarray*}
& &  \displaystyle\int_{[\underline{v}+\epsilon > n, \overline{v}+\epsilon \leq n]}\Big(-|\nabla \underline{v}|^{p-2}\nabla \underline{v}\nabla\overline{v} \frac{(\overline{v}+\epsilon)^{p-1}p}{n^{p-1}} + |\nabla \overline{v}|^p + \frac{(p-1)n^p|\nabla \overline{v}|^p}{(\overline{v}+\epsilon)^p}\Big)dx \\
& & + \displaystyle\int_{[\overline{v}+\epsilon \leq \underline{v}+\epsilon \leq n]} \Big(|\nabla \underline{v}|^{p} - p\Big(\frac{\overline{v}+\epsilon}{\underline{v}+\epsilon}\Big)^{p-1}|\nabla \underline{v}|^{p-2}\nabla \underline{v}\nabla \overline{v} + (p-1)\Big(\frac{\overline{v}+\epsilon}{\underline{v}+\epsilon}\Big)^{p}|\nabla \underline{v}|^p \Big. \\
& & ~~~~~~~~~~~~+ \Big. |\nabla \overline{v}|^{p} - p\Big(\frac{\underline{v}+\epsilon}{\overline{v}+\epsilon}\Big)^{p-1}|\nabla \overline{v}|^{p-2}\nabla \overline{v}\nabla \underline{v} + (p-1)\Big(\frac{\underline{v}+\epsilon}{\overline{v}+\epsilon}\Big)^{p}|\nabla \overline{v}|^p \Big)dx \\
& = & \displaystyle\int_{\Omega}|\nabla \underline{v}|^{p-2}\nabla \underline{v}\nabla \varphi_1dx - \displaystyle\int_{\Omega}|\nabla \overline{v}|^{p-2}\nabla \overline{v}\nabla \varphi_2dx \\
& \leq &  \displaystyle\int_{\Omega} a\Big[\frac{\underline{v}^{-\delta}}{(\underline{v}_n^{\epsilon})^{p-1}} - \frac{\overline{v}^{-\delta}}{(\overline{v}_n^{\epsilon})^{p-1}}\Big][(\underline{v}_n^{\epsilon})^p - (\overline{v}_n^{\epsilon})^p]^+dx + \displaystyle\int_{\Omega} b\Big[\frac{\underline{v}^{\beta}}{(\underline{v}_n^{\epsilon})^{p-1}} - \frac{\overline{v}^{\beta}}{(\overline{v}_n^{\epsilon})^{p-1}}\Big][(\underline{v}_n^{\epsilon})^p - (\overline{v}_n^{\epsilon})^p]^+dx. 
\end{eqnarray*}

Denoting by
\begin{eqnarray*} 
I& = & \displaystyle\int_{[ \overline{v}+\epsilon \leq \underline{v}+\epsilon \leq n]} \Big(|\nabla \underline{v}|^{p} - p\Big(\frac{\overline{v}+\epsilon}{\underline{v}+\epsilon}\Big)^{p-1}|\nabla \underline{v}|^{p-2}\nabla \underline{v}\nabla \overline{v} + (p-1)\Big(\frac{\overline{v}+\epsilon}{\underline{v}+\epsilon}\Big)^{p}|\nabla \underline{v}|^p  \Big. \\
& & ~~~~~~~~~~ + \Big. |\nabla \overline{v}|^{p} - p\Big(\frac{\underline{v}+\epsilon}{\overline{v}+\epsilon}\Big)^{p-1}|\nabla \overline{v}|^{p-2}\nabla \overline{v}\nabla \underline{v} + (p-1)\Big(\frac{\underline{v}+\epsilon}{\overline{v}+\epsilon}\Big)^{p}|\nabla \overline{v}|^p \Big)dx, \end{eqnarray*}
and using the previous inequality together with the Picone's inequality, we have
\begin{eqnarray}\label{12}
0 \leq  I & \leq & \displaystyle\int_{[\underline{v}+\epsilon > n, \overline{v}+\epsilon \leq n]}|\nabla \underline{v}|^{p-1}|\nabla \overline{v}|dx + \displaystyle\int_{\Omega} a\Big[\frac{\underline{v}^{-\delta}}{(\underline{v}_n^{\epsilon})^{p-1}} - \frac{\overline{v}^{-\delta}}{(\overline{v}_n^{\epsilon})^{p-1}}\Big][(\underline{v}_n^{\epsilon})^p - (\overline{v}_n^{\epsilon})^p]^+dx \nonumber\\
& & + \displaystyle\int_{[\underline{v}+\epsilon > n, \overline{v}+\epsilon \leq n]} b\Big[\frac{\underline{v}^{\beta}}{n^{p-1}} - \frac{\overline{v}^{\beta}}{(\overline{v} + {\epsilon})^{p-1}}\Big][n^p - (\overline{v} + {\epsilon})^p]dx \nonumber\\
& & + \displaystyle\int_{[\overline{v}+\epsilon \leq \underline{v}+\epsilon \leq n]} b\Big[\frac{\underline{v}^{\beta}}{(\underline{v} + \epsilon)^{p-1}} - \frac{\overline{v}^{\beta}}{(\overline{v} + {\epsilon})^{p-1}}\Big][(\underline{v} + {\epsilon})^p - (\overline{v} +{\epsilon})^p]dx. 
\end{eqnarray}

Let us consider each one of the integrals in (\ref{12}). The Dominated Convergence Theorem implies that  
\begin{equation}\label{13}
\displaystyle\int_{[\underline{v}+\epsilon > n, \overline{v}+\epsilon \leq n]}|\nabla \underline{v}|^{p-1}|\nabla \overline{v}|dx \stackrel{n \rightarrow \infty}{\longrightarrow} 0.
\end{equation} 
Also, by manipulating in the second integral, we obtain
\begin{equation}\label{14}
 \displaystyle\int_{\Omega} a\Big[\frac{\underline{v}^{-\delta}}{(\underline{v}_n^{\epsilon})^{p-1}} - \frac{\overline{v}^{-\delta
 }}{(\overline{v}_n^{\epsilon})^{p-1}}\Big][(\underline{v}_n^{\epsilon})^p - (\overline{v}_n^{\epsilon})^p]^+dx \leq 0
\end{equation}
 for all $n \in \mathbb{N}$ and  $\epsilon > 0$. To the second last one, the Dominated Convergence Theorem implies again that
\begin{eqnarray}\label{15}
 & & \displaystyle\int_{[\underline{v}+\epsilon > n, \overline{v}+\epsilon \leq n]} b\Big[\frac{\underline{v}^{\beta}}{n^{p-1}} - \frac{\overline{v}^{\beta}}{(\overline{v} + {\epsilon})^{p-1}}\Big][n^p - (\overline{v} + {\epsilon})^p]dx \nonumber \\
& \leq &  \displaystyle\int_{[\underline{v}+\epsilon > n, \overline{v}+\epsilon \leq n]} b\Big[\underline{v}^{\beta}(\underline{v}+\epsilon) + \overline{v}^{\beta}(\overline{v}+\epsilon)\Big]dx \stackrel{n \rightarrow \infty}{\longrightarrow} 0.
\end{eqnarray}

For the last integral, since
$$ b\Big[\frac{\underline{v}^{\beta}}{(\underline{v} + \epsilon)^{p-1}} - \frac{\overline{v}^{\beta}}{(\overline{v} + {\epsilon})^{p-1}}\Big][(\underline{v} + {\epsilon})^p - (\overline{v} +{\epsilon})^p]^+ \leq \Big[\underline{v}^{\beta}(\underline{v}+\epsilon) + \overline{v}^{\beta}(\overline{v}+\epsilon)\Big] \in L^1(\Omega), $$ 
it follows from Fatou's Lemma that
\begin{eqnarray}\label{16}
& & \displaystyle\limsup_{\epsilon \rightarrow 0}\displaystyle\int_{[\overline{v}+\epsilon \leq \underline{v}+\epsilon \leq n]} b\Big[\frac{\underline{v}^{\beta}}{(\underline{v} + \epsilon)^{p-1}} - \frac{\overline{v}^{\beta}}{(\overline{v} + {\epsilon})^{p-1}}\Big][(\underline{v} + {\epsilon})^p - (\overline{v} +{\epsilon})^p]dx \nonumber\\
&  \leq & \displaystyle\int_{[\overline{v}+\epsilon \leq \underline{v}+\epsilon \leq n]} b\Big[\frac{\underline{v}^{\beta}}{\underline{v}^{p-1}} - \frac{\overline{v}^{\beta}}{\overline{v}^{p-1}}\Big][\underline{v}^p - \overline{v}^p]dx \leq 0, ~~ \mbox{for all} ~ n \in \mathbb{N}.
\end{eqnarray}

Hence, going back to (\ref{12}) and using (\ref{13}), (\ref{14}), (\ref{15}) and (\ref{16}), we get 
\begin{eqnarray*}
0 &\leq & \displaystyle\limsup_{\epsilon \rightarrow 0^+}\displaystyle\liminf_{n \rightarrow \infty} I \leq \displaystyle\limsup_{\epsilon \rightarrow 0^+}\displaystyle\liminf_{n \rightarrow \infty} \Big(\displaystyle\int_{\Omega} a\Big[\frac{\underline{v}^{-\delta}}{(\underline{v}_n^{\epsilon})^{p-1}} - \frac{\overline{v}^{-\delta}}{(\overline{v}_n^{\epsilon})^{p-1}}\Big][(\underline{v}_n^{\epsilon})^p - (\overline{v}_n^{\epsilon})^p]^+dx \Big. \nonumber\\
& & + \displaystyle\int_{[\underline{v}+\epsilon > n, \overline{v}+\epsilon \leq n]} b\Big[\frac{\underline{v}^{\beta}}{n^{p-1}} - \frac{\overline{v}^{\beta}}{(\overline{v} + {\epsilon})^{p-1}}\Big][n^p - (\overline{v} + {\epsilon})^p]dx \nonumber\\
& & \Big.+ \displaystyle\int_{[\overline{v}+\epsilon \leq \underline{v}+\epsilon \leq n]} b\Big[\frac{\underline{v}^{\beta}}{(\underline{v} + \epsilon)^{p-1}} - \frac{\overline{v}^{\beta}}{(\overline{v} + {\epsilon})^{p-1}}\Big][(\underline{v} + {\epsilon})^p - (\overline{v} +{\epsilon})^p]dx \Big).   
\end{eqnarray*}  
 
Since are assuming that $(\underline{v} - \overline{v})^+ \neq 0$ and  $ a + b > 0$ hold, it follows from the previous inequality that
$$ 0  \leq  \displaystyle\limsup_{\epsilon \rightarrow 0^+}\displaystyle\liminf_{n \rightarrow \infty} I < 0, $$ which is an absurd. Therefore  $(\underline{v} - \overline{v})^+ = 0$. This ends the proof.
 \fim
\smallskip
 
\noindent\textbf{\textit{Proof of Theorem 1.1 (Uniqueness):}}\label{unicidade}  In any case,  by the Theorem \ref{unicidade1} we get $u \leq v$ and $v \leq u$, which implies $u =v$.
\fim

\section{$W_{loc}^{1,p}(\Omega)$-continuity and local behavior for a solution application}
 
Throughout this section, we are going to assume the assumptions of Theorem 1.1. So, it is well-defined  the solution application $T:(0,\infty) \to W_{loc}^{1,p}(\Omega)$ given by
 \begin{equation}
 \label{1141}
 T(\alpha) = u_\alpha,
 \end{equation}
 where $u_\alpha \in W_{loc}^{1,p}(\Omega)$ is the unique solution of Problem $(L_\alpha)$ given by Theorem 1.1.
 
 Besides this, it is an immediate consequence of Theorem \ref{unicidade1} the below Proposition.
\begin{proposition}\label{P1}
The application $T$ in non-decreasing.
\end{proposition} 

Below, let us prove that $T$ is a ``$W_{loc}^{1,p}(\Omega)$-continuous application", i.e. 
\begin{center}
if $\alpha_n \to \alpha$ in $\mathbb{R}$, then $T(\alpha_n) \to T(\alpha)$ in $W^{1,p}(U)$ for each $U\subset\subset \Omega$ given.
\end{center}

To do this, let us begin stating a sub-supersolution result whose proof follows close arguments as done in Theorem 2.4  by Nguyen  and Schmitt in \cite{LOC}.

\begin{theorem}[Sub and Supersolution Theorem]\label{ts} Suppose that $a$ and $b$ satisfy $(\tilde{h}_0)$ and $\underline{u},\overline{u}\in W^{1,p}_{loc} \cap C(\Omega) \cap  L^{\infty}(\Omega)$ are subsolution and supersolution of problem $(L_{\alpha})$, respectively, with  $0 < \underline{u} \leq \overline{u}$ a.e. in $\Omega$. Then there exists a $u \in W^{1,p}_{loc}\cap L^{\infty}(\Omega)$  satisfying the equation in $ (L_{\alpha})$ with $u \in [\underline{u}, \overline{u}]$.
\end{theorem}

 In what follows, $\Phi_1 \in W_0^{1,p}(\Omega)$ will denote the positive normalized eigenfunction associated to 
\begin{eqnarray}\label{autovalor}
-\Delta_p u = \lambda_1H_1(x){|u|}^{p-2}u   \ \mbox{in} \ \ \Omega, \ \ \Phi_1|_{\partial \Omega} = 0
\end{eqnarray}
 where $H_1(x) := \min\{a(x), b(x)\}\geq 0$ and $\lambda_1 > 0$ will stand for the first eigenvalue of (\ref{autovalor}) (see \cite{AG} and \cite{CUESTA1} for more details about (\ref{autovalor})). If  $(\tilde{h}_0)$ is satisfied, then by 
 \cite{MIGI} one can conclude that \linebreak $\Phi_1 \in C(\overline{\Omega})$. Moreover, if $(\tilde{h}_1)$ holds, then  $\Phi_1$ belongs to the interior of the positive cone in $C^{1}_0(\overline{\Omega})$ (see Corollary 1.1 in \cite{GUEDDA}) and hence we conclude by \cite{VAZ} that for some positive constant, one has
   \begin{eqnarray}\label{4.3}
  Cd(x)\leq \Phi_1(x) \ \  \mbox{in}  \  \Omega ,
   \end{eqnarray}
where $d(x)$ stands for the distance between $x \in \Omega$ and the boundary $\partial\Omega$.

Similarly, defining
$H_2(x) : = \max\{a(x), b(x)\} \geq 0$ and denoting the unique positive solution of \begin{eqnarray}\label{peso}
 -\Delta_p u = H_2(x) \ \ \mbox{in} \  \Omega, \ \ \ u|_{\partial \Omega} = 0
 \end{eqnarray} 
 by $e \in W_0^{1,p}(\Omega)$, it follows from   $(\tilde{h}_0)$ and \cite{MIGI} that $e \in C(\overline{\Omega})$.
 
\begin{lemma}[$T(\alpha)$-behavior for small $\alpha>0$]\label{lema1}
Suppose that  $(\tilde{h}_0)$ is satisfied. Then there exists an $\alpha_0 > 0$ such that $T(\alpha)\in [\underline{u}_{\alpha} ,\overline{u}_{\alpha}]$ for all $\alpha \in (0,\alpha_0)$, where $\underline{u}_{\alpha} := \alpha^q\Phi_1$ and $\overline{u}_{\alpha}:= \alpha^l e^t$ with $q > \frac{1}{p - 1 + \delta}, \  l < \frac{1}{p - 1 + \delta}$ and 
$ t = \frac{p - 1}{p - 1 + \delta}$. In particular, $T(\alpha) \in W_{loc}^{1,p}(\Omega)\cap C(\overline{\Omega})$ for all $\alpha \in (0,\alpha_0)$.
\end{lemma}

\noindent\begin{proof}
Let $\alpha>0$. Since $q > \frac{1}{ p - 1 + \delta}$ hold, we are able 
 to find $\alpha_1 > 0$ such that  $\sup_{\overline{\Omega}} \Phi_1^{p - 1 +\delta} \leq \lambda_1^{-1}\alpha^{1 - q(p-1 + \delta)}$ for all $\alpha \in ( 0, \alpha_1)$. Thus, 
$$ -\Delta_p \underline{u}_{\alpha} \leq \lambda_1\alpha^{q(p - 1)}H
_1
(x) \displaystyle\sup_{\overline{\Omega}}\Phi_1^{p-1} \leq \frac{\alpha^{1 - q\delta}}{\displaystyle\sup_{\overline{\Omega}} \Phi_1^{\delta}}a(x) = \alpha\frac{a(x)}{\underline{u}_{\alpha}^{\delta}} \leq \alpha \Big( \frac{a(x)}{\underline{u}_{\alpha}^{\delta}} + b(x)\underline{u}_{\alpha}^{\beta} \Big) $$
holds true.

To the supersolution,  define $\overline{u}_{\alpha} =  \alpha^le^t$, where $ t= \frac{p-1}{p-1+\delta} $ and $l < \frac{1}{p - 1 + \delta}$. So, we obtain
\begin{eqnarray}\label{1181}
\!\!\!\! \!\!\!\! \displaystyle\int_{\Omega}
 |\nabla\overline{u}_{\alpha}|^{p-2}\nabla\overline{u}_{\alpha} \nabla\varphi dx
\stackrel{0<t<1}{\geq} \displaystyle\int_{\Omega} |\nabla e|^{p-2}\nabla e \nabla \Big[\varphi(\alpha^le^{t-1}t)^{p-1}\Big]dx  = \displaystyle\int_{\Omega} H_2(x)\Big[\varphi(\alpha^le^{t-1}t)^{p-1}\Big]dx
\end{eqnarray}
for all  $ \varphi \geq 0 \ \ \mbox{in}  \ C_c^{\infty}(\Omega)$.
 
Besides this, by using that $l < \frac{1}{p - 1 + \delta}$,  we can choose $\alpha_2 > 0$ such that
$$ t^{p-1}\alpha^{l(p - 1 + \delta) - 1} \geq 1 + \alpha^{l(\beta + \delta)}\displaystyle\sup_{\overline{\Omega}}e^{t(\beta + \alpha)}$$ 
 for all $\alpha \in (0, \alpha_2)$. So by using this inequality in (\ref{1181}), we get $\overline{u}$ is a supersolution for $(L_{\alpha}). $ 

Moreover, by taking $\epsilon > 0$ such that $\epsilon^{p-1}\lambda_1\displaystyle\sup_{\overline{\Omega}}\Phi_1^{p-1} < 1$, we have that $\epsilon \Phi_1 \leq e$. Since $l<q$, there exists $\alpha_3 > 0$ such that  $ \displaystyle\sup_{\overline{\Omega}} \Phi_1^{1-t} \leq \epsilon^t \alpha^{l - q}$ for all $\alpha \in (0, \alpha_3)$. So by taking $\alpha_0 = \min \{\alpha_1,\alpha_2,\alpha_3\}$, we obtain from above informations that 
$ \underline{u}_{\alpha} \leq \overline{u}_{\alpha}.$

 Therefore, it follows from  Theorem \ref{ts} that there exists a  function $u  \in W^{1,p}_{loc} \cap L^{\infty}(\Omega)$ satisfying the equation in  $(L_{\alpha})$ with $u \in [\underline{u}_{\alpha}, \overline{u}_{\alpha}]$ for all $\alpha>0$ small enough.  Since, $ \underline{u}_{\alpha},\overline{u}_{\alpha} \in W^{1,p}_{loc}(\Omega)  \cap C(\overline{\Omega})$, we obtain that $u$  satisfies the boundary condition given in Definition \ref{fronteira}.  So, by the uniqueness claimed in Theorem 1.1,  we conclude that  $u = u_{\alpha}=T(\alpha)$. 
 
Finally, it follows from the hypothesis   $(\tilde{h}_0)$,  the fact that $T(\alpha)\in [\underline{u}_{\alpha} ,\overline{u}_{\alpha}]$ and  Corollary 8.1 in  \cite{MIGI} that  $u_{\alpha} \in C(\Omega)$ for $\alpha>0$ small. As $\underline{u}_{\alpha}$ and $\overline{u}_{\alpha} \in C(\overline{\Omega})$ and  $\underline{u}_{\alpha}\vert_{\partial \Omega} = \overline{u}_{\alpha} \vert_{\partial \Omega} = 0, $ the required regularity follows.
\fim
\end{proof}
\vspace{0.2cm}

Following close arguments as done above, we prove the next Lemma.
\begin{lemma}[$T(\alpha)$-behavior for large $\alpha>0$]
\label{lema2} Assume that  $(\tilde{h}_0)$ is satisfied. Then, there exists $\alpha_{\infty} > 0$ such that $T(\alpha)\in [\underline{u}_{\alpha} ,\overline{u}_{\alpha}]$ for all $\alpha \in (0,\alpha_\infty)$, where $\underline{u}_{\alpha} := \alpha^q\Phi_1$ and $\overline{u}_{\alpha}:= \alpha^l e^t$ with $ q < \frac{1}{p-1 - \beta}, \ l > \frac{1}{p -1 - \beta} $ and $ t = \frac{p-1}{p-1+\delta} $.  In particular, $T(\alpha) \in W_{loc}^{1,p}(\Omega)\cap C(\overline{\Omega})$ for all $\alpha \in (\alpha_\infty, \infty)$.
\end{lemma}

After the above Lemmas and Proposition 3.1, we obtain that
$$T((0,\infty)) \subset  W_{loc}^{1,p}(\Omega)\cap C(\overline{\Omega})$$
when $(\tilde{h}_0)$  holds. Now, we are able to prove the continuity of $T$.
\begin{lemma}
\label{CONT} Suppose $(\tilde{h}_0)$  holds. Then $T$ is a continuous application both in the $ W^{1,p}_{loc}(\Omega)$ topology and  $ C(\overline{\Omega})$ one.
\end{lemma}

\begin{proof} First let us prove the continuity of $T$ in the $ W^{1,p}_{loc}(\Omega)$-topology. Consider $ \alpha_n \rightarrow \alpha > 0$ in $\mathbb{R}$. Then, it follows from  Lemmas \ref{lema1},  \ref{lema2} and monotonicity established in the Proposition  \ref{P1}, that  there exist  $0<\underline{\alpha} < \alpha_0 $ and  $\overline{\alpha} > \alpha_{\infty} $ such that  
\begin{eqnarray}\label{1211}
 \underline{\alpha}^q \Phi_1 = u_{\underline{\alpha} }\le u_{\alpha_n}\le u_{\overline{\alpha}} = \overline{\alpha}^le^t ~~ \mbox{ in } ~~ \Omega, ~ \mbox{for all } n \in \mathbb{N}. 
 \end{eqnarray}

 Take an open set $U \subset \subset \Omega$ and $\xi \in C_c^{\infty}(\Omega)$ such that  $  0 \leq \xi \leq 1$ and $\xi = 1$ in $U$. By using $u_{\alpha_n}\xi^p$ as a test functions in $(L_{\alpha_n})$, we obtain
\begin{eqnarray}\label{11}
\displaystyle \int_{\Omega} |\nabla u_{\alpha_n}|^p \xi^pdx + p\displaystyle\int_{\Omega} |\nabla u_{\alpha_n}|^{p-2}\nabla u_{\alpha_n} \nabla \xi u_{\alpha_n} \xi^{p-1}dx = \alpha_n \displaystyle\int_{\Omega} \Big[ a(x)u_{\alpha_n}^{-\delta+1} + b(x)u_{\alpha_n}^{\beta + 1} \Big]\xi^pdx. 
\end{eqnarray} 

So, it follows from  the boundedness of $(u_n)$ in $L^{\infty}(\Omega)$ and Young's inequality that
\begin{eqnarray}\label{21}
\displaystyle\int_{\Omega} |\nabla u_{\alpha_n}|^{p-2} \nabla u_{\alpha_n} \nabla \xi u_{\alpha_n}\xi^{p-1}dx & \leq & \displaystyle\int_{\Omega} |\nabla u_{\alpha_n}|^{p-1} |\nabla \xi| u_{\alpha_n}\xi^{p-1}dx  \nonumber \\
& \leq & \epsilon \displaystyle\int_{\Omega} \left(|\nabla u_{\alpha_n}|^{p-1}\xi^{p-1}\right)^{\frac{p}{p-1}}dx + C(\epsilon) \displaystyle\int_{\Omega} u_{\alpha_n}^p |\nabla \xi|^p dx \nonumber \\
& \leq & \epsilon \displaystyle\int_{\Omega} \xi^p |\nabla u_{\alpha_n}|^pdx + C(\epsilon),  
\end{eqnarray}
where $C(\epsilon)$ is a cumulative positive constant.

Hence, by using (\ref{1211}) and (\ref{21}) in (\ref{11}) , we obtain 
$$  \displaystyle\int_{U} |\nabla u_{\alpha_n}|^p dx \leq \displaystyle\int_{\Omega} |\nabla u_{\alpha_n}|^p \xi^pdx \leq  C(\epsilon),$$
which implies that $(u_{\alpha_n})$ is bounded in $W_{{loc}}^{1,p}(\Omega)$. So, there exists $u \in W_{loc}^{1,p}(\Omega)$ such that
\begin{equation}
\label{121}
 \left\{
\begin{array}{l}
 u_{\alpha_n} \rightharpoonup  u \ \ \mbox{in}  \  \ W^{1,p}(U) \\
 u_{\alpha_n} \rightarrow u \ \ \ \mbox{in}  \ \ L^{q}(U) \  \mbox{ for all} \ \  1 \leq q < p^* \\
u_{\alpha_n}(x) \rightarrow u(x) \ \ \mbox{almost everywhere in} \ \Omega ,

\end{array}
\right.
\end{equation}
for each  $U \subset \subset \Omega$ given.

By applying \cite{MUR} (see Theorem 2.1), we obtain $$ \nabla u_{\alpha_n} \rightarrow \nabla u, \   \mbox{ in} \  (L^q(\Omega))^N \  \mbox{for any } \  q < p.$$ 
As a consequence, for each $\varphi \in C_c^{\infty}(\Omega)$ given, we get
$$ \displaystyle\int_{\Omega}( |\nabla u_{\alpha_n}|^{p-2}\nabla u_{\alpha_n} - |\nabla u(x)| )^{p-2}\nabla u) \nabla \varphi dx
\rightarrow 0. $$
Moreover,  if $K$  denote the
support of  $\varphi$, we have
$$ \Big|\Big( \frac{a}{u_{\alpha_n}^{\delta}} + bu_{\alpha_n}^{\beta}\Big)\varphi \Big| \leq \Big(\frac{a}{(\underline{\alpha}^q\displaystyle\inf_{K}\Phi_1)^{\delta}} + b \overline{\alpha}^{l\beta}\displaystyle\sup_{K}e^{t\beta}\Big)\|\varphi\|_{\infty} \in L^1(K), $$
thus using the Dominated Convergence Theorem, we get
 $$ \alpha_n \displaystyle\int_{\Omega} \Big( \frac{a}{u_{\alpha_n}^{\delta}} + bu_{\alpha_n}^{\beta}\Big)\varphi dx \longrightarrow \alpha \displaystyle\int_{\Omega} \Big( \frac{a}{u^{\delta}} + bu^{\beta}\Big)\varphi dx. $$  
Hence, we conclude that
$$\displaystyle\int_{\Omega} |\nabla u|^{p-2}\nabla u \nabla \varphi = \alpha \ \displaystyle\int_{\Omega} \Big( \frac{a}{u^{\delta}} + bu^{\beta}\Big)\varphi , \ \ \ \forall \ \varphi \in C_c^{\infty}(\Omega). $$

Since $\underline{\alpha}^q\Phi_1 \leq u_{\alpha_n} \leq \overline{\alpha}^{l}e^t$, we have $\underline{\alpha}^q\Phi_1 \leq u \leq \overline{\alpha}^{l}e^t$. Thus, as $\Phi_1$ and $e \in C(\overline{\Omega})$, we have
$0 \leq (u-\epsilon)^+ \leq (\overline{\alpha}^{l}e^t - \epsilon)^+,$
that is,  $(u-\epsilon)^+ \in W_0^{1,p}(\Omega)$ for each $\epsilon>0$ given, that is $u$ satisfies the boundary condition of Definition \ref{fronteira}.  So, by applying the  uniqueness of $W_{loc}^{1,p}(\Omega)$-solutions to Problem $(L_{\alpha})$ claimed in Theorem \ref{T6}, we have that $u = u_{\alpha}$.

For the $C(\overline{\Omega})$-continuity,  it follows from  (\ref{1211}) that the sequence $(u_{\alpha_n}) $ is  bounded in $C^{\alpha}(K)$ for some $\alpha \in (0,1)$ and  each compact $K \subset \Omega$ given. So, it follows from 
Arzel$\grave{a}$-Ascoli's Theorem and (\ref{121}), that $u_{\alpha_n} \to u$ in $C(\Omega)$. Besides this, by using (\ref{1211}), we obtain that $u \in C(\overline{\Omega})$  and  $u_{\alpha_n} \rightarrow u$ in $C(\overline{\Omega})$.
\fim
\end{proof}

\section{Multiplicity of $W^{1,p}_{loc}(\Omega)$-solutions for $(P_{\lambda})$}
Now we are able to prove Theorem 1.2. Before that, we will introduce the  applications $G:D(G)\subset W^{1,p}_{loc}(\Omega) \to [0,\infty)$ and $H:(0,\infty) \to (0,\infty)$ defined by
\begin{equation}\label{114} G(u) = \Big(\displaystyle\int_{\Omega} g(x,u)dx\Big)^r ~~~ \mbox{and} ~~~~ H(\alpha) = \alpha G(T(\alpha)),
\end{equation} 
where $D(G) = \{0\leq u \in W^{1,p}_{loc}(\Omega)~/~ G(u)<\infty\}$. 

Besides these, let us consider the system 
\begin{equation}\label{arcoya21}
\left\{\begin{array}{l}
\displaystyle\int_{\Omega} |\nabla  
u|^{p-2}\nabla u\nabla \varphi dx = {\alpha} \displaystyle\int_{\Omega} \Big(a(x)u^{-\delta} + b(x)u^{\beta}\Big)\varphi dx \\
\alpha G(u)= \lambda ,
\end{array}\right.
\end{equation}
remind that
$$\Sigma = \{(\lambda, u) \in (0,\infty)\times C(\overline{\Omega})~/~u \in W^{1,p}_{loc}(\Omega)~\mbox{is solution of}~ (P_{\lambda})\}$$
and denote by
$$\Sigma' = \{(H(\alpha), u_{\alpha}) \in (0,\infty)\times C(\overline{\Omega})~/ ~\alpha \in (0, \infty)~\mbox{and }u_\alpha\in W_{loc}^{1,p}(\Omega)~\mbox{is a solution of }(L_\alpha)\}.$$

As a consequence of Lemma 3.3, we can prove the next result. 

\begin{lemma}\label{lema4.2} Suppose one of the following item holds:
\begin{itemize}
\item[$(i)$] $(\tilde{h}_0)$ is satisfied and $g \in C(\overline{\Omega}\times[0,\infty), (0, \infty))$;
\item[$(ii)$] $(\tilde{h}_1)$, $g \in C(\overline{\Omega}\times(0,\infty), (0, \infty))$ and  $\displaystyle\lim_{t \rightarrow 0^+}g(x, t)t^{\theta} = f(x) \geq 0$ uniformly in $\overline{\Omega}$, for some $f \in C(\overline{\Omega})$ and $0 < \theta < 1$. 
\end{itemize}
Then $T\big((0,\infty) \big) \subset D(G)$ and, in particular, $H$  is well-defined. Besides this, $H$  is a continuous function.
\end{lemma}
\begin{proof}
Take  $\alpha > 0$. It follows from Lemmas 3.1, 3.2 and the monotonicity established in 
Proposition 3.1, we can find $0 < \underline{\alpha} = \underline{\alpha}(\alpha) < \alpha_0$ and $\overline{\alpha} = \overline{\alpha}(\alpha) > \alpha_{\infty}$ such that 
\begin{equation}
\label{118}
\underline{\alpha}^q\Phi_1 \leq u_{\alpha} \leq \overline{\alpha}^le^t~\mbox{in }\Omega,
\end{equation}
where $\Phi_1$ and $e$ are the solutions of the Problems (\ref{autovalor}) and (\ref{peso}), respectively.

First, let us assume  $(ii)$ holds. So, by choosing an $\epsilon' ,t_0>0$ sufficiently small such that $\overline{\alpha}^le^t < t_0$ for all $x \in \Omega_{\epsilon'} = \{ x \in \Omega: \mbox{dist}(x, \partial \Omega) < \epsilon'\} $, we obtain from (\ref{4.3}), (\ref{118}) and hypothesis $(ii)$, that $0 < g(x,u_{\alpha}) \leq \tilde{f}(x)  d(x)^{-\theta}$ in $\Omega_{\epsilon'} $ for some $0<\tilde{f} \in C(\overline{\Omega})$.

Since $\theta < 1$,  it follows from \cite{LAZ} and the previous  inequality that $ g(x, u_{\alpha}) \in L^1(\Omega_{\epsilon'}),$   which proves that $H$ is well-defined in this case.

About the case  $(i)$, the result follows directly from the  fact that $e$ is a bounded function.  So, in both cases, we showed that $T(\alpha) \in D(G)$ for each $\alpha>0$ given.

To show the continuity, consider $\alpha_n\rightarrow \alpha>0$. By an analogous argument as in first part, we can conclude that in any case there exists a $ h(x) \in L^1(\Omega)$ such that $g(x, u_{\alpha_n}) \leq h(x) $, for all $x \in \Omega$ and $n \in \mathbb{N}$. Thus, the continuity follows from the Lemma 3.3 and Convergence Dominated Theorem.
\fim
\end{proof}
\\

After this Lemma, by using the uniqueness claimed in Theorem \ref{T6}, we obtain the next one.

\begin{lemma}\label{lema4.1} Let $\lambda>0$. Then Problem $(P_\lambda)$ admits a $W_{loc}^{1,p}(\Omega)$-solution if, and only if, there exist $(\alpha,u)=(\alpha_\lambda,u_\lambda) \in (0,\infty) \times W^{1,p}_{loc}(\Omega)$ solution of $(\ref{arcoya21})$. In particular, Problem $(P_\lambda)$ admits a $W_{loc}^{1,p}(\Omega)$-solution if, and only if, $\lambda \in H((0,\infty))$.
\end{lemma}

As a rereading of the above Lemma and a consequence of Lemma \ref{CONT},  we conclude that 
$$\Sigma = \{(H(\alpha), u_{\alpha}) \in (0,\infty)\times C(\overline{\Omega})~/ ~\alpha \in (0, \infty)~\mbox{and }u_\alpha\in W_{loc}^{1,p}(\Omega)~\mbox{is a solution of }(L_\alpha)\}$$
is the {\it continuum} of solutions to Problem $(P_\lambda)$ given by a curve.
\vspace{0.2cm}

\noindent\textbf{\textit{Proof of Theorem $1.2$-Completed:}} 

\begin{itemize}
\item[1-a)] 
Firstly,  note that by the continuity of g  and  Lemma \ref{lema1}, we have that \linebreak $\displaystyle\lim_{\alpha \rightarrow 0}H(\alpha) = 0$. We will split the prove in two cases:
\begin{itemize}
\item[$i)$] \textit{case $1$: $r \geq 0$.} If $\theta_1 \geq 0$, then by taking  $U \subset \subset \Omega$ and using  (\ref{4}) together with  Lemma \ref{lema2}, we obtain that 
\[\int_\Omega g(x, u_{\alpha})dx\geq \int_{U} g(x, u_{\alpha})dx \geq {C}{\alpha^{-\theta_1l}}\]
holds for all $\alpha$ sufficiently  large. So, choosing $l \in \Big(\frac{1}{p - 1 - \beta}, \frac{1}{\theta_1 r}\Big)$, we get 
\[H(\alpha)=\alpha\left(\int_\Omega g(x, u_{\alpha})dx\right)^r\geq C \alpha^{1- r\theta_1l}\rightarrow \infty\mbox{ as } \alpha\rightarrow \infty.\]

Suppose now that $ \theta_1 <0 $. In this case, by using (\ref{4}) and Lemma \ref{lema2} again, we get 
$$ H(\alpha) \geq  C\alpha^{1-q\theta_1r} \rightarrow \infty \mbox{ as} \ \alpha \rightarrow \infty. $$
So, in both cases we have $H(\alpha) \to \infty $ as $\alpha \to \infty$.

\item[$ii$)] \textit{case $2$: $r < 0$.} {Consider the case $\theta_1 \geq 0$.   By the hypothesis (\ref{4}) and the continuity of $g$, we obtain }
$ \displaystyle\int_{\Omega} g(x,u_\alpha) dx  \leq C$,
that is, $H(\alpha) \geq C \alpha \rightarrow \infty$ as $ \alpha \rightarrow \infty$.

Analogously, when $\theta_1 < 0$, we obtain by the Lemma \ref{lema2} and the hypothesis (\ref{4})
 $$H(\alpha) \geq C\alpha \Big(1 + \alpha^{-\theta_1l}\Big)^r = C\Big(\alpha^{\frac{1}{r}} + \alpha^{\frac{1}{r} -\theta_1l}\Big)^r $$
 showing that $H(\alpha) \rightarrow \infty$ as $\alpha \rightarrow \infty$ by choosing   $l <{1}/{\theta_1 r}$ in Lemma \ref{lema2}. 
Hence, in all cases we have $H(\alpha) \rightarrow 0$ as $\alpha \rightarrow 0$ and $H(\alpha) \rightarrow \infty$ as $\alpha \rightarrow \infty$. Since $H$ is continuous (see Lemma \ref{lema4.1}),  our claim  follows.
\end{itemize}
To finish the proof, it just remains  to show the behavior of the {\it continuum} $\Sigma$ at $\lambda=0$ and $\lambda=\infty$.
To $\lambda=0$, let us take $\epsilon > 0$ and define $\delta = \displaystyle\inf_{[\epsilon, \infty)} H(\alpha)$. So, it follows from the Lemma \ref{CONT} and  $H(\alpha)\to \infty$ as $\alpha \to \infty$ that $\delta > 0$ and $(0, \delta) \subset H((0, \epsilon))$, that is, for each $\lambda_n \in (0, \delta)$ given there exists an $\alpha_n \in (0, \epsilon)$ such that $H(\alpha_n) = \lambda_n$. So,  if $\lambda_n \rightarrow 0$, then $\alpha_n \rightarrow 0$, which implies by the Lemma \ref{lema1} that $\| u_{\alpha_n}\|_{\infty} \rightarrow 0$. To $\lambda =\infty$, define $m = \displaystyle\max_{[0, M]} H(\alpha)$ for each  $M> 0$ given, so  $m < \infty$ and $(m, \infty) \subset H((M, \infty))$, that is, for each $\lambda_n \in (m, \infty)$, there exists $\alpha_n \in (M, \infty)$ such that $\lambda_n = H(\alpha_n)$. Hence,  if $\lambda_n \rightarrow \infty$, then $\alpha_n \rightarrow \infty$ and so  by using  Lemma \ref{lema2}, we obtain that $\|u_{\alpha_n}\|_{\infty} \rightarrow \infty. $ See picture Fig. 1.

\item[1-$b$)] Initially suppose that $ r > 0 $.  In this case $ \theta_1> 0 $, because $\theta_1 r > p-1-\beta > 0.$ By the hypotheses (\ref{4}) we obtain  $g(x, t) \leq C_1t^{-\theta_1}$ for all $t \geq t_0$. Remembering that  $ Cd (x) \leq \Phi_1(x) $ in $ \Omega $ holds because we are assuming  $(\tilde{h}_1)$, it follows from Lemma \ref{lema2} that
\begin{eqnarray*} \displaystyle\int_{\Omega} g(x, u_{\alpha})dx & = & \displaystyle\int_{[u_{\alpha} \geq t_0]} g(x, u_{\alpha})dx + \displaystyle\int_{[u_{\alpha} < t_0]} g(x, u_{\alpha})dx \\ & \leq & C\Big(\displaystyle\int_{[u_{\alpha} \geq t_0]} g(x, u_{\alpha})dx + \displaystyle\int_{[C\alpha^qd(x) < t_0]}1dx\Big)  \\ & \leq & C\Big( \alpha^{-q\theta_1} + \alpha^{-qN}\Big) \end{eqnarray*} 
for $\alpha>0$ large enough. So,
$$H(\alpha) \leq \alpha C\Big(  \alpha^{-q\theta_1} + \alpha^{-qN} \Big)^r = C \Big( \alpha^{\frac{1}{r} -q\theta_1} +  \alpha^{\frac{1}{r} -qN}\Big)^r \rightarrow 0 ~\mbox{as}~\alpha \rightarrow \infty$$
since we are  taking $q \in \Big(\frac{1}{r\max\{\theta_1, N\}}, \frac{1}{p-1-\beta}\Big). $

Let us now consider the case where $ r <0 $. In this case, $ \theta_1 <0 $, because we are using $\theta_1r > p-1-\beta > 0$ again. Hence, in an  analogous way to that done above, we can prove that  $H(\alpha) \leq C \alpha^{1 - q\theta_1 r}$
for sufficiently large $\alpha>0$. Thus, by taking $q \in \Big(\frac{1}{\theta_1r}, \frac{1}{p-1-\beta}\Big), $ we get $H(\alpha) \rightarrow 0 $ as $\alpha \rightarrow \infty.$ 

In any case, we proved that $\displaystyle\lim_{\alpha \rightarrow \infty} H(\alpha) = 0$. On the other hand, $ H (\alpha) \rightarrow 0 $ as $\alpha \rightarrow 0$, therefore by  taking $\lambda^* = \displaystyle\sup_{\mathbb{R}^+} H(\alpha)$, the claims follows.

About the behavior of $ \Sigma$. Letting $(\lambda,u) \in  \Sigma$ we have that  $\lambda \leq  \lambda^*$. Since $\displaystyle\lim_{\alpha \rightarrow 0} H(\alpha) = \displaystyle\lim_{\alpha \rightarrow \infty} H(\alpha) = 0$,  we get $(0, \delta) \subset H((0, \epsilon) \cap H((M, \infty))$ for each $\epsilon > 0$ small and $M > 0$ large, where $0 < \delta = \displaystyle\min_{[\epsilon, M]} H(\alpha)$. Thus, for each $\lambda_n \in (0, \delta)$ there exists $\alpha^1_n \in (0, \epsilon)$ and $\alpha_n^2 \in (M, \infty)$ such that $\lambda_n = H(\alpha_n^1) = H(\alpha_n^2)$. So, if $\lambda_n \rightarrow 0 $ imply that $\alpha_n^1 \rightarrow 0 $ and $\alpha_n^2 \rightarrow \infty$, which lead us to conclude that  $\|u_{\alpha_n^1}\|_{\infty} \rightarrow 0 $ and $\|u_{\alpha_n^2}\|_{\infty} \rightarrow \infty $ after to use Lemmas \ref{lema1} and \ref{lema2} again. See Fig. 2.  

\item[2-$a$)]
 Initially assume that $ r > 0 $.  In this case $ \theta_2> 0 $, because $\theta_2 r > p-1 + \delta > 0.$  Then, by taking $ l < \frac{1}{p - 1 + \delta}  $, $\alpha > 0$ sufficiently  close to $0$, using the hypothesis (\ref{5}) and Lemma \ref{lema1}, we get
\begin{eqnarray}\label{6}
H(\alpha)\geq  C\alpha\left(\displaystyle\int_{\Omega} \frac{1}{\alpha^{l\theta_2}e(x)^{t\theta_2}}dx\right)^r
= C\alpha^{1- r\theta_2l}
\end{eqnarray}
for some $C > 0$  constant. 
By choosing  $ l>{1}/{\theta_2r}  $ in Lemma \ref{lema1}, we obtain from (\ref{6})  that $H(\alpha) \rightarrow + \infty$ as $\alpha \rightarrow 0^+. $ In the same way, when $r, \theta_2 < 0$ by the hypotheses (\ref{5}) and the Lemma \ref{lema1}, we obtain $H(\alpha) \rightarrow + \infty$ as $\alpha \rightarrow 0^+. $

On the other hand, by following the same idea as used to prove the item 1-a), we can verify that $H(\alpha)\rightarrow \infty$ when $\alpha\rightarrow \infty$.   Thus, by considering $\lambda^*= \displaystyle\inf_{\alpha\in\mathbb{R}^+}H(\alpha)$, the result follows. 
\item[2-$b$)] By the same arguments as used to prove the items $1-b)$ and $2-a)$, we can verify that 
 $H(\alpha) \stackrel{\alpha \rightarrow \infty}{\longrightarrow} 0$ and $H(\alpha)  \stackrel{\alpha \rightarrow 0^+}{\longrightarrow} \infty$, so the result follows again. These ends the proof of Theorem $1.2$.
\fim
\end{itemize}

Similarly to the cases  $1-a)$ and $ 1-b)$, we are able to verify that the \textit{continuum} $\Sigma$ behaves as in the figures Fig.3 (item $2-a)$) and Fig. 4 (item $2-b)$), respectively. 

\begin{remark} Despite of our objective in this paper is to present situations of how to break the uniqueness of $W_{loc}^{1,p}(\Omega)$-solution to local Problem $(L_\alpha)$ by the introduction of a non-local term, we note that is still possible to obtain uniqueness of solutions to non-local Problem 
$(P_\lambda)$. For instance, when $g(t) = t^{\gamma}$ for $t>0$ with either $\{ \gamma>0 ~ \mbox{and }~r>0\}$ or $\{ -1<\gamma <0~ \mbox{and }~r<0\}$.
\end{remark}

\section{Appendix}

In this section, lets sketch  the proof of existence in  Theorem 1.1. For this, we will consider the following auxiliary problem:
\begin{equation}\label{n-lsp} \left\{
\begin{array}{l}
  -\Delta_p u=  \frac{a_n(x)}{(u + \frac{1}{n})^{\delta}}  + b_n(x)u^{\beta}  ~\mbox{in } \Omega,\\
    u>0    ~\mbox{in } \partial\Omega,~~
    u>0    ~\mbox{on }  \Omega
 \end{array}
\right.
\end{equation}
where $ a_n(x) = \min\{a(x), n \}$ and $b_n(x) = \{b(x), n\}$, with $n \in \mathbb{N}$.

\begin{lemma}\label{L1}
For each $n\in \mathbb{N}$, the problem $(\ref{n-lsp})$ admits a solution $u_n \in W_0^{1,p}(\Omega) \cap C^{1, \alpha}(\overline{\Omega})$. Furthermore, for each compact set $ K \subset \subset \Omega$ there  exists  $c_{K} > 0$ such that 
$ u_n \geq c_K > 0 \ \mbox{in}  \ K,$  for all $ n  \in \mathbb{N}. $
\end{lemma}
\begin{proof}
For each $ v \in L^p(\Omega)$, we claim that there exists a unique function $ \omega \in W_0^{1,p}(\Omega)$ solution  of
\begin{eqnarray}\label{3.1} -\Delta_p \omega = \frac{a_n(x)}{(|v| + \frac{1}{n})^{\delta}} + b_n(x)|v|^{\beta}. 
\end{eqnarray}
In fact, consider the functional $ J: W_0^{1,p}(\Omega) \rightarrow \mathbb{R}$ defined by 
$$ J(\omega) = \frac{1}{p}\displaystyle\int_{\Omega} |\nabla \omega|^p dx - \displaystyle\int_{\Omega} \frac{a_n(x)}{(|v| + \frac{1}{n})^{\delta}} \omega dx - \displaystyle\int_{\Omega} b_n(x)|v|^{\beta}\omega dx.$$

We can easily verify that $J$ is differentiable, strictly convex and coercive. Hence $J$ admits a unique critical point, that is, (\ref{3.1}) admits a solution.

Denoting by $ S: L^p(\Omega) \rightarrow L^p(\Omega)$ the operator, which associates to each $v \in L^p(\Omega)$ the unique solution $w = S(v) \in L^p(\Omega)$ of (\ref{3.1}), one can prove that $S$ is a continuous and compact operator. Furthermore, if $  \omega = \lambda S(\omega)$  for some  $\lambda \in (0, 1]$ and $\omega \in W_0^{1,p}(\Omega)$, then by Poincar\'e's and 
H\"{o}lder inequalities 
\begin{eqnarray*}
 {\|\omega\|_{p}}^p & \leq &  C\lambda^p\displaystyle\int_{\Omega} |\nabla S(\omega)|^p dx= C\lambda^p\displaystyle\int_{\Omega} \Big[\frac{a_n}{(\frac{1}{n}+|\omega|)^{\delta}}S(\omega) + b_n(x)|\omega|^{\beta} S(\omega)\Big]dx \\
 & \leq & C\lambda^{p-1}\displaystyle\int_{\Omega}\Big( n^{1+\delta}|\omega| + n|\omega|^{\beta +1} \Big) dx \leq C\Big( \|\omega\|_{p} + \|\omega\|_{p}^{\beta+1} \Big), 
 \end{eqnarray*}
where $C > 0$ is a cumulative constant. 
 
Thus, by the previous inequality, there exists a positive constant $R$, independent of $\lambda$ and $\omega$, such that $\|\omega\|_{p} \leq R$. So, by the Schaefer Fixed Point Theorem, there exists a $u_n \in W_0^{1,p}(\Omega)$ such that $S(u_n) = u_n$.

Note that $ {a_n}{(|t| + \frac{1}{n})^{-\delta}} + b_n |t|^{\beta} \leq C(1 + |t|^{\beta})$, so  by \cite{LIB}  we have that $u_n \in C^{1, \alpha}(\overline{\Omega}), $ for some $\alpha \in (0,1)$. Furthermore, $  {a_n}{(|u_n| + \frac{1}{n})^{-\delta}} + b_n |u_n|^{\beta} \geq 0 $, thus $ u_n \geq 0$  which by \cite{VAZ} implies $ u_n > 0$ in $\Omega$. Therefore, $ u_n$ is a positive solution of (\ref{n-lsp}).

Beside this, suppose that $\tilde{u}_1$ is a solution of 
\begin{eqnarray}\label{3.2}
  -\Delta_p u  = \frac{a_1(x)}{(1 + u)^{\delta}}  \ \mbox{in} \ \Omega, \ u > 0 \ \mbox{in} \ \Omega  \ \mbox{and}  \ u  = 0 \ \mbox{on} \ \Omega. 
  \end{eqnarray}
Taking $ (\tilde{u}_1 - u_n)^+ \in W_0^{1,p}(\Omega)$ as a test function in (\ref{n-lsp}) and in (\ref{3.2}), we  get
\begin{eqnarray*}
 \displaystyle\int_{\Omega} (|\nabla \tilde{u}_1| + |\nabla u_n|)^{p-2}|\nabla( \tilde{u}_1 - u_n)^+|^2 dx &\leq & C\displaystyle\int_{\Omega} \Big( |\nabla \tilde{u}_1|^{p-2}\nabla \tilde{u}_1 - |\nabla u_n|^{p-2}\nabla u_n\Big) \nabla (\tilde{u}_1 - u_n)^+ dx \\
 & \leq & C\displaystyle\int_{\Omega} a_1 \Big[ \frac{1}{(1+ \tilde{u}_1)^{\delta}} - \frac{1}{(1 + u_n)^{\delta}} \Big](\tilde{u}_1 - u_n)^+ dx \leq 0.
 \end{eqnarray*} 
Therefore, $( \tilde{u}_1 - u_n)^+ = 0,$ that is, $\tilde{u}_1 \leq u_n $ in $\Omega$.  

Finally, by \cite{LIB} we concluded that $\tilde{u}_1 \in C^{1, \alpha}(\overline{\Omega})$ for some $\alpha \in (0,1)$. Therefore, using this and the positivity of $\tilde{u}_1$ in $\Omega$, the last part of the Lemma follows.
\fim
\end{proof}
\vspace{0.2cm}

\noindent\textbf{\textit{Proof of Theorem 1.1 ( Existence-Conclusion):}}
Consider a sequence $(\Omega_k)$ of open sets in $\Omega$ such that $\Omega_k \subset \Omega_{k+1}$ and $\bigcup\limits_{k} \Omega_k = \Omega$, and define $ \delta_k = \displaystyle\inf_{\Omega_k} \tilde{u}_1  > 0,$ where $\tilde{u}_1$ is the solution of (\ref{3.2}). Take $\varphi = (u_n - \delta_1)^+$ as a test function in (\ref{n-lsp}). If  $(h)_1$ holds, then we have
\begin{eqnarray*} 
\displaystyle\int_{u_n > \delta_1} |\nabla u_n|^{p} dx & \leq &  \displaystyle\int_{u_n > \delta_1} \Big(\frac{a}{u_n ^{\delta - 1}} + bu_n^{\beta + 1}\Big)dx \\
& \leq & \|a\|_{(\frac{p^*}{1-\delta})'}\Big(\displaystyle\int_{u_n > \delta_1} u_n^{p^*}dx\Big)^{\frac{1-\delta}{p^*}} + \|b\|_{(\frac{p^*}{\beta + 1})'} \Big( \displaystyle\int_{u_n > \delta_1} u_n^{p^*} dx\Big)^{\frac{\beta + 1}{p^*}}\\
& \leq & C\Big[1 + \Big(\displaystyle\int_{u_n > \delta_1} |\nabla u_n|^{p}dx\Big)^{\frac{1-\delta}{p}} + \Big( \displaystyle\int_{u_n > \delta_1} |\nabla u_n|^{p}dx\Big)^{\frac{\beta + 1}{p}} \Big].
\end{eqnarray*}

In a similar way, we obtain
\begin{eqnarray*} 
\displaystyle\int_{u_n > \delta_1} |\nabla u_n|^{p}dx & \leq &  \displaystyle\int_{u_n > \delta_1} \Big(\frac{a}{u_n ^{\delta - 1}} + bu_n^{\beta + 1} \Big)dx \leq \delta_1^{1 - \delta}\displaystyle\int_{\Omega} a dx  + \|b\|_{(\frac{p^*}{\beta + 1})'} \Big( \displaystyle\int_{u_n > \delta_1} u_n^{p^*}dx\Big)^{\frac{\beta + 1}{p^*}}  \\
& \leq & C\Big[ 1 +  \Big( \displaystyle\int_{u_n > \delta_1} |\nabla u_n|^{p}dx\Big)^{\frac{\beta + 1}{p}}\Big]
\end{eqnarray*}
is true if either $(h)_2$ or $(h)_3$ holds.

Therefore,  we conclude in both the cases that  $\displaystyle\int_{\Omega_1} |\nabla u_n|^{p}dx$ will be bounded. Furthermore, since $ (u_n - \delta_1)^+ \in W_0^{1,p}(\Omega)$ we have
\begin{eqnarray*}
\displaystyle\int_{\Omega_1} {u_n }^p dx & \leq & \displaystyle\int_{u_n > \delta_1} {u_n }^p dx\leq C\Big[ 1+  \displaystyle\int_{\Omega} {(u_n - \delta_1)^+}^p dx\Big] \\
& \leq & C\Big[ 1+ \displaystyle\int_{u_n > \delta_1} |\nabla u_n|^p  dx\Big] \leq  C. 
\end{eqnarray*}
Thus, we conclude that $(u_n)$ is bounded in $W^{1,p}(\Omega_1)$. Hence, there exists $u_{\Omega_1} \in W^{1,p}(\Omega_1)$ and a subsequence $(u_{n_j^1})$ of $(u_n)$ such that 
$$
 \left\{
\begin{array}{l}
u_{n_j^1} \rightharpoonup u_{\Omega_1} \ \ \mbox{weakly in } \ \ W^{1,p}(\Omega_1) \\ \mbox{and strongly in } \ \ L^q(\Omega_1) \ \ \mbox{for} \ 1 \leq q < p^* \\
u_{n_j^1} \rightarrow u_{\Omega_1}  \ \ 
a.e  \ \mbox{in } \ \Omega_1.
\end{array}
\right.
$$

Proceeding as above, we can obtain subsequences $(u_{n_j^k})$ of $(u_n)$, where $(u_{n_j^{k + 1}}) \subset (u_{n_j^{k}})$, and functions $ u_{\Omega_k} \in W^{1,p}(\Omega_k)$ such that 
$$
 \left\{
\begin{array}{l}
u_{n_j^k} \rightharpoonup u_{\Omega_k} \ \ \mbox{weakly in } \ \ W^{1,p}(\Omega_k) \ \ \mbox{and strongly in }  \ L^p(\Omega_k) \ \mbox{for} \ 1 \leq q < p^*,  \\
u_{n_j^k} \rightarrow u_{\Omega_k} \ 
a.e \ \mbox{in}   \ \Omega_k.
\end{array}
\right.
$$

By construction, $ u_{\Omega_{k+1}}\Big|_{\Omega_k} = u_{\Omega_k}. $ Defining $$u = \left\{
\begin{array}{l} u_{\Omega_1} \ \ \mbox{in} \ \ \Omega_1, \\ u_{\Omega_{k+1}} \ \ \mbox{in} \ \ \Omega_{k+1}\backslash \Omega_k,\end{array}
\right.$$ then $u \in  W_{loc}^{1,p}(\Omega)$. Furthermore, by following close arguments as done in \cite{ELVES}, we are able to show that $u$ is a positive solution of (\ref{lsp}). 

To finish the proof, let us note that when $ \delta \leq 1$, by taking $u_n$ as test function in (\ref{n-lsp}) and following  similar arguments  as done above one can conclude that $(u_n)$ is bounded in $W_0^{1,p}(\Omega).$ In this case, $u$ defined as above belongs to $W_0^{1,p}(\Omega). $
\fim

\end{document}